\documentclass[12pt, a4paper]{amsart}

\usepackage[hmargin=30mm, vmargin=25mm, includefoot, twoside]{geometry}
\usepackage[bookmarksopen=true]{hyperref}

\usepackage{amsfonts,amssymb,verbatim}
\usepackage{latexsym}
\usepackage{mathrsfs}
\usepackage{stmaryrd}
\usepackage{xspace}
\usepackage{enumerate, paralist}
\usepackage{graphicx}
\usepackage[all]{xy}
\usepackage{extarrows}
\usepackage{tabu}
\usepackage{accents}
\usepackage{tensor}

\usepackage[usenames,dvipsnames]{color}

\usepackage{txfonts, pxfonts}

\usepackage{amsthm}
\usepackage{amsmath}

\usepackage[tbtags]{mathtools}

\newcommand{\ubar}[1]{\underaccent{\bar}{#1}} 

\newtheorem{thmprime}{Theorem}[section]
\newtheorem{thmpprime}{Theorem}[section]

\newtheorem{thm}{Theorem}[section]
 \newtheorem{cor}[thm]{Corollary}
 \newtheorem{lem}[thm]{Lemma}
 \newtheorem{prop}[thm]{Proposition}

\newtheorem{introthm}{Theorem}

 \newtheorem{introdefn}[introthm]{Definition}

\numberwithin{equation}{section}

 \theoremstyle{definition}
  \newtheorem{defn}[thm]{Definition}

 \theoremstyle{remark}
 \newtheorem{rem}[thm]{Remark}
  \newtheorem{ex}[thm]{Example}
   
\newtheorem*{claim*}{Claim}

\newcommand{\mhyphen}{\operatorname{-}}
\newcommand{\variable}{\,\mhyphen\,}
\newcommand{\symdiff}{{\, \resizebox{0.8 em}{!}{$\triangle$}\, }}

\def\B{\mathfrak B}
\def\K{\mathfrak K}

\def\G{\mathcal G}
\def\s{\mathrm{s}}
\def\r{\mathrm{r}}
\def\dyn{\mathrm{dyn}}
\def\q{\mathrm{q}}
\def\dy{\mathrm{dy}}
\def\d{\mathrm{d}}
\def\Ad{\mathrm{Ad}}
\def\uq{\mathrm{uq}}
\def\u{\mathrm{u}}
\def\AFS{\mathrm{AFS}}

\def\N{\mathbb N}
\def\C{\mathbb C}

\def\SOT{\mathrm{SOT}}
\def\sdeg{s\text{-}\mathrm{deg}}

\def\CC{\mathbf{C}}

\def\G{\mathcal{G}}
\def\Gz{\mathcal{G}^{(0)}}

\def\RR{\mathcal{R}}
\def\S{\mathcal{S}}

\def\P{\mathfrak{P}}

\def\RN{\mathfrak{R}}

\def\KK{\mathcal{K}}

\begin{document}

\title{Asymptotic expansion for groupoids and Roe type algebras}
\author{Xulong Lu}
\author{Qin Wang}
\author{Jiawen Zhang}

\address[Xulong Lu]{Research Center for Operator Algebras, School of Mathematical Sciences, East China Normal University, 500 Dongchuan Road, Shanghai, 200241, China.}
\email{51255500057@stu.ecnu.edu.cn}

\address[Q. Wang]{Research Center for Operator Algebras, School of Mathematical Sciences, East China Normal University, Shanghai, 200241, China.}
\email{qwang@math.ecnu.edu.cn}

\address[Jiawen Zhang]{School of Mathematical Sciences, Fudan University, 220 Handan Road, Shanghai, 200433, China.}
\email{jiawenzhang@fudan.edu.cn}

\thanks{QW is partially supported by NSFC (No.  12171156), and the Science and Technology Commission of Shanghai Municipality (No. 22DZ2229014). JZ is supported by National Key R\&D Program of China 2022YFA1007000 and NSFC12422107.}

\begin{abstract}
In this paper, we introduce a notion of expansion for groupoids, which recovers the classical notion of expander graphs by a family of pair groupoids and expanding actions in measure by transformation groupoids. We also consider an asymptotic version for expansion and establish structural theorems, showing that asymptotic expansion can be approximated by domains of expansions. On the other hand, we introduce dynamical propagation and quasi-locality for operators on groupoids and the associated Roe type algebras. Our main results characterise when these algebras possess block-rank-one projections by means of asymptotic expansion, which generalises the crucial ingredients in previous works to provide counterexamples to the coarse Baum-Connes conjecture.
\end{abstract}

\date{\today}

\maketitle


\section{Introduction}\label{sec:intro}

Over the last few decades, the phenomenon of expansion has been discovered and extensively studied across various branches of mathematics. For instance in graph theory, the expansion phenomenon leads to the notion of expander graphs, which plays an important role not only in pure and applied mathematics but also in theoretical computer science (see the excellent survey article \cite{Lub12} and the references therein). In dynamical systems, the expansion phenomenon leads to the notion of expanding actions in measure, which turns out to be equivalent to the classic notion of spectral gap for measure-preserving actions (see \cite{Vig19}) and numerous examples have been discovered (\emph{e.g.}, \cite{BdS16, BG08, GJS99}).


Recently, an asymptotic version of expansion was introduced in different areas of mathematics (\cite{Li2019QuasilocalAA, asymptoticexpansionandstrongergodicity}), which is more stable under small perturbations and hence leads to important applications in operator algebras and higher index theory (\cite{structure, li2022markovianroealgebraicapproachasymptotic, faucris.282438222}). A crucial step therein is structural type theorems, showing that objects with asymptotic expansion can be approximated by those with expansion. In dynamical systems, asymptotic expansion also provides a new quantitative viewpoint on the classic notion of strong ergodicity, which was introduced in \cite{CW80, Sch80, Sch81} in relation with Ruziewicz problem, Kazhdan’s Property (T) and amenability.


In this paper, we aim to generalise and unify the theory of asymptotic expansion from different areas (including graph theory and dynamical systems) in the language of groupoids. Groupoids provide a framework encompassing both groups and spaces. They arise naturally in a variety of research areas such as dynamical systems, topology and geometry, geometric group theory and operator algebras, building bridges between all these areas of mathematics (see, \emph{e.g.}, \cite{brownronald1987, Higgins1971}).


To achieve this, we introduce the notion of expansion and asymptotic expansion for groupoids, generalising both (asymptotic) expander graphs and (asymptotically) expanding actions in measure, and establish structural type theorems in the groupoid setting. Furthermore, we introduce two classes of operator algebras associated to the dynamics of groupoids, generalising the classical Roe and quasi-local algebras from higher index theory. Our main results show that the existence of certain projection operators in these operators algebras characterise asymptotic expansion of groupoids, which provide a unified approach to results in \cite{structure, li2022markovianroealgebraicapproachasymptotic, faucris.282438222} and allow a boarder range of examples and applications.

We will now give a more detailed overview of this work.


\subsection{Expansion and asymptotic expansion}
To motivate our notion of asymptotic expansion, let us first recall the notion of asymptotically expanding actions in measure from \cite{asymptoticexpansionandstrongergodicity}. 

Let $\rho:\Gamma\curvearrowright X$ be a measure-class-preserving action of a countable group $\Gamma$ acting on a probability space $(X,\mu)$, and $\ell_\Gamma$ be a proper length function on $\Gamma$ (\emph{i.e.}, for each $L>0$, the ball $B_L\coloneqq \{\gamma \in \Gamma ~|~ \ell_\Gamma(\gamma) \leq L\}$ is finite). The action is called \emph{asymptotically expanding in measure $\mu$} if for any $\alpha \in (0,\frac{1}{2}]$ there exists $C_\alpha, L_\alpha>0$ such that for any measurable $A \subseteq X$ with $\alpha \leq \mu(A) \leq \frac{1}{2}$, we have $\mu(B_{L_\alpha} \cdot A)> (1+C_\alpha)\mu(A)$. When the functions $\alpha \mapsto L_\alpha$ and $\alpha \mapsto C_\alpha$ can taken to be constant functions, then the action is called \emph{expanding in measure $\mu$}.

To generalise the above to the setting of groupoids (see Section \ref{Basic notions for groupoid} for basic notions of groupoids), we consider the transformation groupoid. Recall that for the action $\rho:\Gamma\curvearrowright X$, the \emph{transformation groupoid} $X \rtimes \Gamma$ has source $\s(x,\gamma) = \gamma^{-1}x$, range $\r(x,\gamma)=x$ and inverse $(x,\gamma)^{-1} = (\gamma^{-1}x, \gamma^{-1})$. The length function $\ell_\Gamma$ naturally gives rise to a length function $\ell$ on $X \rtimes \Gamma$ by $\ell(x, \gamma)\coloneqq\ell_{\Gamma}(\gamma)$ for $x\in X$ and $\gamma \in \Gamma$. 

This leads to the following setting of our paper: Let $\G$ be a groupoid with a length function $\ell$ (see Definition \ref{length function}), and the unit space $\Gz$ is equipped with a measure $\mu$ on some $\sigma$-algebra $\RR$. To abstract measure-class-preserving transformations, we introduce the following: A bisection $K \subseteq \G$ is called \emph{admissible} if its source $\s(K)$ and range $\r(K)$ are measurable, the induced bijection 
\begin{equation}\label{EQ:alpha_K}
\tau_K\coloneqq\r|_K \circ (\s|_K)^{-1}: \s(K)\rightarrow \r(K)
\end{equation}
is a measure-class-preserving measurable isomorphism and the length $\ell(K)\coloneqq\sup\{\ell(x) ~|~ x\in K\}$ is finite. Moreover, a subset $K \subseteq \G$ is called \emph{decomposable} if $K=\bigcup_{i=1}^{N} K_{i}$ for $N \in \N$ and admissible bisections $K_i$. Note that for transformation groupoids, any measurable subset (with respect to the product structure) with finite length is decomposable, while unfortunately this does not hold in general.

Now we are in the position to introduce our notion of asymptotic expansion:


%

\begin{introdefn}[Definition \ref{expanding} and \ref{asymptotically expanding}]\label{introdefn:asymptotic expansion}
Let $\mathcal{G}$ be a groupoid with a length function $\ell$, and $(\Gz,\RR,\mu)$ be a probability measure space. We say that $\G$ is \emph{asymptotically expanding (in measure $\mu$)} if for any $\alpha \in (0,\frac{1}{2}]$, there exists a decomposable $K_\alpha \subseteq \G$ with $\ell(K_\alpha) < \infty$ such that for any $A \in \RR$ with $\alpha \leq \mu(A) \leq \frac{1}{2}$, then $\mu(\r(K_\alpha \cdot A)\setminus A)> C_\alpha \mu(A)$. If $K_\alpha \equiv K$, then we say that $\G$ is \emph{expanding (in measure $\mu$)}.
\end{introdefn}

It is clear that in the case of transformation groupoids, this recovers the notion of (asymptotically) expanding actions (see Section \ref{Group action}). On the other hand, when considering pair groupoids, a family version of Definition \ref{introdefn:asymptotic expansion} recovers the notion of (asymptotic) expander graphs (see Section \ref{main thm for Uniform Roe algebra} for details).

We establish the following structure theorem for asymptotic expansion, showing that it can be approximated by domains of expansion:

%


\begin{introthm}[Theorem \ref{exhaustion of expansion with radon controlled} and \ref{thm:Markov structure theorem}]\label{introthm structure theorem}
Let $\mathcal{G}$ be a groupoid with a length function $\ell$, and $(\Gz,\RR,\mu)$ be a probability measure space. Then the following are equivalent:
	\begin{enumerate}
		\item $\mathcal{G}$ is asymptotically expanding in measure.
		\item $\mathcal{G}^{(0)}$ admits an exhaustion by domains of expansion with bounded ratio.
		\item $\mathcal{G}^{(0)}$ admits an exhaustion by domains of Markov expansion with bounded ratio.
	\end{enumerate}
\end{introthm}



Roughly speaking, here ``domain of expansion'' is a measurable subset $Y$ of $\Gz$ such that the reduction $\G^Y_Y$ is expanding (see Definition \ref{domain of asymptotic expansion} and Remark \ref{rem:reduction for domain}), and ``exhaustion'' means a sequence of domains $Y_n$ such that $\mu(\Gz \setminus Y_n) \to 0$. Here we also construct a Markov kernel for each  domain (see Definition \ref{defn:Markov kernel for K}) and consider a Markovian version of expansion, which will play a key role later to produce certain projection operators. Details can be found in Section \ref{ssec:structure thm} and \ref{ssec:Markov structure theorem}.

Theorem \ref{introthm structure theorem} provides a unified approach to recover the structure results in \cite{structure, li2022markovianroealgebraicapproachasymptotic, faucris.282438222}, and is more general to allow other types of examples (see Section \ref{Examples of measurable groupoids}). Even for asymptotic expander graphs and asymptotically expanding actions, our proof simplifies the original ones in a systematic way.

\subsection{Dynamical propagation and quasi-locality}\label{introsubsection finite dynamical propagatio n and quasi-local}

Now we introduce two classes of operator algebras associated to groupoids, which encode dynamical information and have their root in higher index theory. 

Recall that for a metric space, there are two associated operator algebras, called the Roe algebra and the quasi-local algebra. They were essentially introduced by Roe in his pioneering work on higher index theory \cite{Roe1988AnIT}, where he discovered that higher indices of elliptic differential operators on open manifolds belong to $K$-theories of the Roe algebra. More recently, Engel showed in \cite{Engel2015IndexTO,Engel2019RoughIT} that $K$-theories of the quasi-local algebra contains higher indices of more general elliptic pseudo-differential operators on open manifolds. 

To compute their $K$-theories, a programmatic and practical approach is to consult the coarse Baum-Connes conjecture \cite{BCH94}, a central conjecture in higher index theory and closely related to other conjectures like the Novikov conjecture and the Gromov-Lawson conjecture. Unfortunately, counterexamples were discovered in \cite{Higson2002, structure, li2022markovianroealgebraicapproachasymptotic, sawicki2017warped} using (asymptotic) expanders and warped cones of (asymptotically) expanding actions.  
Due to their importance, Roe and quasi-local algebras have been extensively studied and several variants have also been introduced (see \cite{Bao2021StronglyQA, delaat2023dynamicalpropagationroealgebras, quasilocalityforetalegroupoid, li2022markovianroealgebraicapproachasymptotic, ozawa2024embeddingsmatrixalgebrasuniform, ZhangquasilocalityandpropertyA, Yu00} and a recent excellent monograph \cite{willett2020higher}).


Inspired by these works, we introduce the following notions of dynamical propagation and quasi-locality for operators on groupoids. Note that in the context of transformation groupoids and pair groupoids, the following recovers the original (dynamical) Roe and quasi-local algebras (see Section \ref{Group action} and \ref{main thm for Uniform Roe algebra}).

\begin{introdefn}[Definition \ref{finite dynamical propagation} and \ref{dynamical quasi local}]\label{introdefinition of uniformly dynamical finite propagation}
Let $\mathcal{G}$ be a groupoid with a length function $\ell$, and $(\Gz,\RR,\mu)$ be a measure space. For an operator $T \in \B(L^2(\Gz,\mu))$, we say:
\begin{enumerate}
 \item $T$ has finite \emph{dynamical propagation} if there exists a decomposable $K\subseteq \G$ with $\ell(K) < \infty$ such that for any $A,B \in \RR$ with $\mu(\r(K \cdot A)\cap B)=0$, then $\chi_A T\chi_B=0$;
 \item $T$ is \emph{dynamically quasi-local} if for any $\epsilon > 0$, there exists a decomposable  $K_\epsilon \subseteq \G$ with $\ell(K_\epsilon) < \infty$ such that for any $A,B \in \RR$ with $\mu(\r(K_\epsilon \cdot A)\cap B)=0$, then $\|\chi_A T\chi_B\| < \epsilon$. 
\end{enumerate}
The \emph{dynamical Roe algebra of $\G$} is the norm closure of all operators with finite dynamical propagation, denoted by $\CC^*_{\dyn}(\mathcal{G})$. The \emph{dynamical quasi-local algebra of $\G$} is the set of all dynamically quasi-local operators, denoted by $\CC^*_{\dyn,\q}(\mathcal{G})$.
\end{introdefn}

Recall that in graph theory and dynamical systems, a key consequence of asymptotic expansion is that the associated Roe type algebras possess block-rank-one projections (called the ghost projections), which is crucial to provide counterexamples to the coarse Baum-Connes conjecture (see \cite{structure, li2022markovianroealgebraicapproachasymptotic, faucris.282438222}). In the current context of groupoids, we study dynamical propagation and quasi-locality of block-rank-one projections and establish a more general result, which unifies and simplifies the previous ones.

More precisely, we first consider a rank-one projection $P \in \B(L^2(\Gz,\mu))$ and a unit vector $\xi \in L^2(\Gz,\mu)$ in the range of $P$. This induces a probability measure $\nu$ on $\Gz$ given by $\d \nu(x) \coloneqq |\xi(x)|^2 \d \mu(x)$ for $x\in \Gz$, called the \emph{associated measure}. 

The following is our main result:


\begin{introthm}[Theorem \ref{thm:main thm ave proj general}]\label{introthm for main thm for family of block rank one projection}
Let $\mathcal{G}$ be a groupoid with a length function $\ell$, and $(\Gz,\RR,\mu)$ be a measure space. Let $P \in \B(L^2(\Gz,\mu))$ be a rank-one projection, and $\nu$ the associated probability measure on $\Gz$. Then the following are equivalent:
\begin{enumerate}
 \item $P \in \CC^*_{\dyn}(\mathcal{G})$;
 \item $P \in \CC^*_{\dyn,\q}(\mathcal{G})$;
 \item $\G$ is asymptotically expanding in measure $\nu$.
\end{enumerate}
\end{introthm}

The proof of Theorem \ref{introthm for main thm for family of block rank one projection} is quite involved, similar to the outline in the context of graph theory and dynamical systems. Roughly speaking, we firstly reduce to the case of averaging projection by changing measures, and prove that this is equivalent to asymptotic expansion. Then applying the structure theorem (Theorem \ref{introthm structure theorem}), we obtain an exhaustion by domain of Markov expansion. Thanks to the theory of Markov kernels, on each domain we obtain a rank-one projection in the dynamical Roe algebra by functional calculus. Finally, these projections converge in norm to the averaging projection, which concludes the proof. 

After establishing Theorem \ref{introthm for main thm for family of block rank one projection}, we further consider a family version by chasing parameters. Then we obtain a family version of Theorem \ref{introthm for main thm for family of block rank one projection} (Theorem \ref{cor:cor to thm general case prime}), which characterises when the dynamical Roe and quasi-local algebras possess block-rank-one projections. This recovers the most technical results in \cite{structure, li2022markovianroealgebraicapproachasymptotic, faucris.282438222}, and our proof simplifies the original ones in a systematic way.


Finally, we apply our results to several examples, including the known results on asymptotic expanders and asymptotically expanding actions. Moreover, we study new examples including groupoid actions on fibred spaces, the HLS groupoid and graph groupoids, and establish corresponding results therein.

\subsection{Organization}\label{Organization}
The paper is organised as follows. In Section \ref{Preliminaries}, we recall necessary background knowledge in groupoids and Markov kernels. In Section \ref{sec:asymptotic expansion}, we introduce the notion of expansion and asymptotic expansion for groupoids and establish the structure result (Theorem \ref{introthm structure theorem}). Section \ref{Dynamical concepts on groupoids} is devoted to introducing dynamical propagation and quasi-locality, and relate asymptotic expansion to the quasi-locality of the averaging projection (Proposition \ref{asym expansion and quasi local}, a special case of part of Theorem \ref{introthm for main thm for family of block rank one projection}). Then in Section \ref{sec:Main results}, we accomplish the proof of Theorem \ref{introthm for main thm for family of block rank one projection}, making use of Theorem \ref{introthm structure theorem}. Finally in Section \ref{Examples of measurable groupoids}, we explain in detail how our results recover the original ones for asymptotic expanders (by a family version of pair groupoids) and asymptotically expanding actions (by transformation groupoids), and provide new examples including groupoid actions, the HLS groupoid and graph groupoids.


\subsection*{Acknowledgement} We would like to thank Prof. Kang Li and Prof. Federico Vigolo for helpful suggestions after reading an early draft of this preprint.

\section{Preliminaries}\label{Preliminaries}
In this section, we recall the notions of groupoids and Markov kernels.

\subsection{Basic notions for groupoids}\label{Basic notions for groupoid}

Recall that a \emph{groupoid} $\mathcal{G}$ is a small category consists of a set $\mathcal{G}$, a subset $\mathcal{G}^{(0)}$ called the \emph{unit space}, two maps $\s,\r:\mathcal{G}\rightarrow \mathcal{G}^{(0)}$ called the \emph{source} and the \emph{range} maps, respectively, a composition law $(\gamma_1,\gamma_2) \mapsto \gamma_1\gamma_2$ for $(\gamma_1,\gamma_2)\in \mathcal{G}^{(2)}\coloneqq \{(\gamma_1,\gamma_2)\in \mathcal{G}\times \mathcal{G} ~|~\s(\gamma_1)=\r(\gamma_2)\}$ and an inverse map $\gamma \mapsto \gamma^{-1}$. There operations satisfy a couple of axioms, including the associativity law and the fact that $\mathcal{G}^{(0)}$ act as units. 
For $A,B\subseteq \mathcal{G}$, we denote
\begin{itemize}
    \item $A^{-1}\coloneqq\{\gamma^{-1} ~|~ \gamma\in A\}$,
    \item $A \cdot B\coloneqq\{\gamma\in \mathcal{G} ~|~ \gamma=\gamma_1\gamma_2~\text{where}~\gamma_1\in A,\gamma_2 \in B\text{ with }\s(\gamma_1)=\r(\gamma_2)\}$.
\end{itemize}
We say that $A$ is \emph{symmetric} if $A=A^{-1}$, and \emph{unital} if $\Gz \subseteq A$. A subset $A\subseteq \mathcal{G}$ is called a \emph{bisection} if the restrictions of $\s,\r$ to $A$ are injective and recall from (\ref{EQ:alpha_K}) that we have the induced bijection $\tau_A\coloneqq\r|_A \circ (\s|_A)^{-1}: \s(A)\rightarrow \r(A)$.

We also need the notion of length functions on groupoids:

\begin{defn}\label{length function}
Let $\mathcal{G}$ be a groupoid. A \emph{length function} on $\G$ is a map $\ell:\G \rightarrow \mathbb{R}^{+}$ such that $\ell(\mathcal{G}^{(0)})=\{0\}$ and
\begin{enumerate}
    \item $\ell(\gamma^{-1})=\ell(\gamma)$ for all $\gamma\in \mathcal{G}$;
    \item $\ell(\gamma_1\gamma_2) \leq \ell(\gamma_1)+\ell(\gamma_2)$ for all $\gamma_1, \gamma_2 \in \G$ with  $\s(\gamma_1)=\r(\gamma_2)$.
\end{enumerate}
\end{defn}

Given a length function $\ell$ on a groupoid $\G$ and $n\in \N$, denote the closed ball
\begin{equation}\label{EQ:Kn}
B_n\coloneqq\{\gamma \in \mathcal{G} ~|~ \ell(\gamma)\leq n\}.
\end{equation}
Clearly $\{B_n\}_{n\in \N}$ is a sequence of increasing subsets in $\G$ satisfying:
\begin{itemize}
\item $\Gz \subseteq B_0$ and $B_n$ is symmetric for $n\in \mathbb{N}$;
\item $B_n \cdot B_m \subseteq B_{n+m}$ for $n,m \in \N$;
\item $\mathcal{G}=\bigcup_{n\in \N} B_n$.
\end{itemize}
Conversely, given an increasing sequence $\{B_n\}_{n\in \N}$ of subsets in $\G$ satisfying the three conditions above, we can construct a function $\tilde{\ell}$ on $\G$ by
\begin{equation}\label{EQ:tilde ell}
\tilde{\ell}(\gamma)\coloneqq\inf\left\{n\in \N ~|~ \gamma\in B_n\right\} \quad \text{for} \quad \gamma \in \G.
\end{equation}
It is routine to check that $\tilde{\ell}$ is a length function on $\G$. Moreover, it is easy to see:

\begin{lem}\label{length function2}
Given a length function $\ell$ on a groupoid $\G$, let $\{B_n\}_{n\in \N}$ be the sequence from Equation (\ref{EQ:Kn}) and define the length function $\tilde{\ell}$ on $\G$ using (\ref{EQ:tilde ell}). Then $\ell$ and $\tilde{\ell}$ are \emph{coarsely equivalent} in the sense that for any $r> 0$, we have:
\[
\sup\{\tilde{\ell}(\gamma) ~|~ \ell(\gamma)\leq r\} < \infty\quad\text{and}\quad\sup\{\ell(\gamma) ~|~ \tilde{\ell}(\gamma)\leq r\}< \infty.
\]
\end{lem}

%
%

Throughout the paper, let $\G$ be a groupoid with length function $\ell$, and $\Gz$ equipped with a (not necessarily finite) measure $\mu$ on some $\sigma$-algebra $\RR$. 
Motivated by the case of group actions as explained in Section \ref{sec:intro}, we introduce the following:


\begin{defn}\label{admissible}
Let $\mathcal{G}$ be a groupoid with a length function $\ell$, and $(\Gz,\RR,\mu)$ be a measure space. A bisection $K \subseteq \G$ is called \emph{admissible} if $\s(K),\r(K) \in \RR$, $\tau_K: \s(K)\rightarrow \r(K)$ is a measure-class-preserving measurable isomorphism, and $\ell(K)\coloneqq\sup\{\ell(x) ~|~ x\in K\}$ is finite. 
\end{defn}

We record the following elementary observation, whose proof is straightforward and hence omitted.

\begin{lem}\label{inverse,union,composition,intersection still admissible}
If $K,K_1,K_2 \subseteq \G$ are admissible bisections, then $K^{-1}$ and $K_1\cdot K_2$ are also admissible bisections.
\end{lem}

Based on Definition \ref{admissible}, we further introduce the following, which is also motivated from the case of group actions  as explained in Section \ref{sec:intro}.

\begin{defn}\label{decomposable}
Let $\mathcal{G}$ be a groupoid with a length function $\ell$, and $(\Gz,\RR,\mu)$ be a measure space. A subset $K \subseteq \G$ is called \emph{decomposable} if $K=\bigcup_{i=1}^{N} K_{i}$ for $N \in \N$ and admissible bisections $K_i$. 

Given a unital symmetric decomposable $K \subseteq \G$, a decomposition $K=\bigcup_{i=1}^N K_i$ is called \emph{unital} if $K_i = \Gz$ for some $i$ and \emph{symmetric} if there exists a bijection $\sigma:\{1,2,\cdots,N\} \to \{1,2,\cdots,N\}$ such that $K_i^{-1} = K_{\sigma(i)}$. We say that $K$ is \emph{$N$-decomposable} if there exists a unital symmetric decomposition $K=\bigcup_{i=1}^N K_i$ for admissible $K_i$. 
\end{defn}

Applying Lemma \ref{inverse,union,composition,intersection still admissible}, we obtain the following:

\begin{lem}\label{inverse,union,composition still decomposable}
\begin{enumerate}
 \item If $K$ is $N$-decomposable with $\ell(K) \leq L$, then $K^{-1}$ is $N$-decomposable with $\ell(K^{-1}) \leq L$.
 \item If $K_i$ is $N_i$-decomposable with $\ell(K_i) \leq L_i$ for $i=1,2$, then $K_1 \cup K_2$ is $(N_1+N_2)$-decomposable with $\ell(K_1 \cup K_2) \leq \max\{L_1, L_2\}$ and $K_1\cdot K_2$ is $N_1N_2$-decomposable with $\ell(K_1 \cdot K_2) \leq L_1 + L_2$.
\end{enumerate}
\end{lem}

The following easy observation is useful to define (asymptotic) expansion for groupoids later.

\begin{lem}\label{basic assumprop2}
If $K \subseteq \G$ is decomposable and $A \subseteq \Gz$ is measurable, then $\r(K\cdot A)$ is measurable.
\end{lem}

\begin{proof} 
By definition, we can write $K=\bigcup_{i=1}^N K_i$ for admissible bisection $K_i$. Then:
\[
\r(K\cdot A)=\r\left(\bigcup_{i=1}^{N} K_{i}\cdot A\right)=\bigcup_{i=1}^{N}\tau_{K_{i}}(A\cap \s(K_{i})),
\]
which is measurable since each $\tau_{K_i}$ is a measurable isomorphism on $\s(K_i)$.
\end{proof}

Finally, note that the notion of admissible bisections and decomposability remain the same if we change the length function up to coarse equivalence.

\subsection{Markov kernel}
Here we recall a few elementary properties of reversible Markov kernels, which will be used later to establish our main results. We refer to the first chapters of \cite{Rev75} for more background and details.

\begin{defn}\label{def of Markov kernel}
	Let $\S$ be a $\sigma$-algebra on a set $X$. A \emph{Markov kernel} on the measurable space $(X,\S)$ is a function $\Pi\colon X\times\S \to [0,1]$ such that:
	\begin{enumerate}
		\item for every $x\in X$, the function $\Pi(x,\variable)\colon\S \to[0,1]$ is a probability measure;
		\item for every $A\in \S$, the function $\Pi(\variable,A)\colon X\to[0,1]$ is $\S$-measurable.
	\end{enumerate}
\end{defn}

The associated \emph{Markov operator} $\P$ is a linear operator on the space of bounded $\S$-measurable functions, defined by
\begin{equation}\label{EQ:Markov operator defn}
\P f(x)\coloneqq  \int_{X}f(y)\Pi(x,\d y).
\end{equation}

\begin{defn}\label{defn:markov boundary}
	Given a measure $\mu$ on $(X,\S)$ and $A\in\S$, the \emph{($\mu$-)size of the boundary} of $A$ (with respect to $\Pi$) is defined as
	\[
	|\partial_\Pi (A)|_\mu\coloneqq\int_A\Pi(x,X\setminus A)\d \mu(x).
	\]
\end{defn}

We are only interested in the special case of reversible Markov kernel:

\begin{defn}
A Markov kernel $\Pi$ is called \emph{reversible} if there exists a measure $m$ on $(X,\S)$ such that
\[
\int_X f(x)\P g(x)\d m(x)=\int_X \P f(x) g(x)\d m(x)
\]
for every pair of measurable bounded functions $f,g: X\to\mathbb{R}$. The measure $m$ is called a \emph{reversing measure} for $\Pi$ (note that $m$ need not be unique in general). In this case, we also say that $\Pi$ is a reversible Markov kernel \emph{on $(X,m)$}.
\end{defn}

Given a reversing measure $m$ on $X$,
the Markov operator $\P$ can be regarded as a bounded self-adjoint operator on $L^2(X,m)$ with norm $\|\P\|\leq 1$. 
Define the \emph{Laplacian} of $\Pi$ as $\Delta\coloneqq 1-\P$, which is a positive self-adjoint operator whose spectrum is contained in $[0,2]$.

Let $\Pi$ be a reversible Markov kernel on $(X,m)$, where $m$ is a finite measure. Then all constant functions on $X$ belong to $L^2(X,m)$ and are fixed by $\P$. It follows that $\|\P\|=1$ and $1$ belongs to the spectrum of $\P$. Denote the orthogonal complement of the constant functions in $L^2(X,m)$ by $L^2_0(X,m)$, \emph{i.e.}, 
\[
L^2_0(X,m)\coloneqq \left\{f\in L^2(X,m) ~\middle|~ \int_X f(x)\d m(x)=0\right\}.
\]
Note that $L^2_0(X,m)$ is $\P$-invariant and that the spectrum of the restriction of $\P$ on $L^2_0(X,m)$ is contained in $[-1,1]$. We denote the supremum of this spectrum by $\lambda\in \mathbb{R}$. We make the following definition:

\begin{defn}\label{spectral gap markov kernel}
	A reversible Markov kernel on a finite measure space $(X,m)$ is said to have a \emph{spectral gap} if $\lambda<1$. 
\end{defn}

On the other hand, we recall the notion of Cheeger constant as follows:
\begin{defn}\label{defn:Cheeger.constant.Markov}
	The \emph{Cheeger constant} for a reversible Markov kernel $\Pi$ on a finite measure space $(X,m)$ is defined to be
	\[
	\kappa\coloneqq\inf\left\{\frac{|\partial_\Pi (A)|_m}{m(A)}\;\middle|\; A\in\S,\ 0<m(A)\leq\frac 12 m(X)\right\}.
	\]
\end{defn}

Consequently, we have the following significant result relating the spectral gap with the Cheeger constant from \cite{lawler1988bounds} (see also \cite{li2022markovianroealgebraicapproachasymptotic}):

\begin{thm}[{\cite[Theorem 2.1]{lawler1988bounds}}]\label{thm:spectral.characterisation.markov.exp}
	Let $\Pi$ be a reversible Markov kernel on $(X,m)$ where $m$ is finite. Then
	\[
	\frac{\kappa^2}{2} \leq 1-\lambda\leq 2\kappa.
	\]
\end{thm}

\section{Asymptotic expansion in measure and structure theorems}\label{sec:asymptotic expansion}

In this section, we introduce the notion of expansion and asymptotic expansion for groupoids and then establish their structure theory. This generalises both (measured) asymptotic expanders in \cite{structure, Li2019QuasilocalAA, faucris.282438222} and asymptotic expansion in measure for group actions in \cite{asymptoticexpansionandstrongergodicity, li2022markovianroealgebraicapproachasymptotic}. 

Again we always assume that $\G$ is a groupoid with a length function $\ell$, and $\Gz$ is equipped with a measure $\mu$ on some $\sigma$-algebra $\RR$. In this section, we require $\mu$ to be finite. For simplicity, further assume that $\mu$ is a \textbf{probability measure}. Note that all the results in this section are also available for finite measures by rescaling.

\subsection{Expansion and asymptotic expansion}\label{Expansion and asymptotic expansion on Borel groupoids}
We firstly introduce the following:

\begin{defn}\label{expanding}
Let $\mathcal{G}$ be a groupoid with a length function $\ell$, and $(\Gz,\RR,\mu)$ be a probability measure space. We say that $\G$ is \emph{expanding (in measure $\mu$)} if there exist $C,N,L>0$ and a unital symmetric\footnote{Since we are only interested in the existence of $K$, we can always make it unital and symmetric.} $N$-decomposable $K \subseteq \G$ with $\ell(K) \leq L$ such that for any $A \in \RR$ with $0 < \mu(A) \leq \frac{1}{2}$, then $\mu(\r(K \cdot A)\setminus A)> C \mu(A)$. 
In this case, we also say that $\G$ is $(C,N,L)$-\emph{expanding}.
\end{defn}

The inequality ``$\mu(\r(K \cdot A)\setminus A)> C \mu(A)$'' in the definition above makes sense since $\r(K \cdot A)$ is measurable thanks to Lemma \ref{basic assumprop2}.


We also consider the following asymptotic version of Definition \ref{expanding}:

\begin{defn}\label{asymptotically expanding} 
Let $\mathcal{G}$ be a groupoid with a length function $\ell$, and $(\Gz,\RR,\mu)$ be a probability measure space. We say that $\G$ is \emph{asymptotically expanding (in measure $\mu$)} if for any $\alpha \in (0,\frac{1}{2}]$, there exist $C_\alpha, N_\alpha, L_\alpha>0$ and a unital symmetric $N_\alpha$-decomposable $K_\alpha \subseteq \G$ with $\ell(K_\alpha) \leq L_\alpha$ such that for any $A \in \RR$ with $\alpha \leq \mu(A) \leq \frac{1}{2}$, then $\mu(\r(K_\alpha \cdot A)\setminus A)> C_\alpha \mu(A)$. 
The functions $\alpha \mapsto C_\alpha$, $\alpha \mapsto N_\alpha$ and $\alpha \mapsto L_\alpha$ are called \emph{expansion parameters}
of $\G$ (which are not unique).
\end{defn}

To establish the structure theorem, we also need the following notion:

\begin{defn}\label{domain of asymptotic expansion}
Let $\mathcal{G}$ be a groupoid with a length function $\ell$, and $(\Gz,\RR,\mu)$ be a probability measure space. A measurable subset $Y \subseteq \Gz$ is called a \emph{domain of asymptotic expansion} if for any $\alpha \in (0,\frac{1}{2}]$, there exist $C_\alpha, N_\alpha, L_\alpha>0$ and a unital symmetric $N_\alpha$-decomposable $K_\alpha \subseteq \G$ with $\ell(K_\alpha) \leq L_\alpha$ such that for any measurable $A \subseteq Y$ with $\alpha\mu(Y) \leq \mu(A) \leq \frac{1}{2}\mu(Y)$, we have:
\[
\mu\big((\r(K_\alpha \cdot A)\setminus A)\cap Y\big)> C_\alpha \mu(A).
\] 
The functions $\alpha \mapsto C_\alpha$, $\alpha \mapsto N_\alpha$ and $\alpha \mapsto L_\alpha$ are called \emph{expansion parameters} for $Y$. 

If $C_\alpha \equiv C$, $N_\alpha \equiv N$, $L_\alpha \equiv L$ and $K_\alpha \equiv K,$ then we also say that $Y$ is a \emph{domain of $(C,N,L)$-expansion} (or simply a \emph{domain of $(N,L)$-expansion} or a \emph{domain of expansion)}.
\end{defn}

\begin{rem}\label{rem:reduction for domain}
It is easy to see that $Y$ is a domain of asymptotic expansion \emph{if and only if} the reduction $\G_Y^Y\coloneqq \s^{-1}(Y) \cap \r^{-1}(Y)$ is asymptotically expanding in Definition \ref{asymptotically expanding}. Here we choose the former since it is more convenient for our proofs.
\end{rem}


In the following, we collect several useful facts for domains of (asymptotic) expansion, generalising those in \cite{asymptoticexpansionandstrongergodicity,faucris.282438222}. They will play an important role to prove the structure result later.

\begin{lem}\label{extend to beta}
Let $Y\subseteq \mathcal{G}^{(0)}$ be a domain of asymptotic expansion. Then for any $\alpha \in (0,\frac{1}{2}]$ and $\beta \in [\frac{1}{2},1)$, there exist $C,N,L>0$ and a unital symmetric $N$-decomposable $K \subseteq \mathcal{G}$ with $\ell(K) \leq L$ such that for any measurable $A\subseteq Y$ with $\alpha \mu(Y) \leq \mu(A) \leq \beta \mu(Y)$, we have:
\[
\mu\big((\r(K \cdot A)\setminus A)\cap Y\big)> C \mu(A). 
\]
Here $C,N,L$ only depend on $\alpha,\beta$ and expansion parameters for $Y$.
\end{lem}

\begin{proof}
For $\beta\in [\frac{1}{2},1)$, set $\alpha'=\frac{1-\beta}{2} \in (0,\frac{1}{4}]$. 
By assumption, there exist $C_{\alpha'}, L_{\alpha'}, N_{\alpha'}>0$ and a unital symmetric $N_{\alpha'}$-decomposable $K_{\alpha'} \subseteq \G$ with $\ell(K_{\alpha'}) \leq L_{\alpha'}$ such that for any measurable $A \subseteq Y$ with $\alpha'\mu(Y) \leq \mu(A) \leq \frac{1}{2}\mu(Y)$, we have:
\begin{equation}\label{EQ:expansion beta}
\mu\big((\r(K_{\alpha'} \cdot A)\setminus A)\cap Y\big)> C_{\alpha'} \mu(A).
\end{equation}

Given measurable $A\subseteq Y$ with $\mu(A)\in [\frac{1}{2} \mu(Y),\beta \mu(Y)]$, let 
\[
B=Y\setminus \r(K_{\alpha'} \cdot A).
\] 
Then $\mu(B)\leq \mu(Y\setminus A)\leq \frac{1}{2} \mu(Y)$. Now we divide into two cases:

If $\mu(B)< \alpha' \mu(Y)$, then $\mu\big(\r(K_{\alpha'} \cdot A)\cap Y\big)> (1-\alpha')\mu(Y) \geq \frac{1+\beta}{2\beta} \mu(A)$. Hence we obtain $\mu\big((\r(K_{\alpha'}\cdot A)\setminus A)\cap Y\big)> \frac{1-\beta}{2\beta} \mu(A)$.

If $\mu(B)\geq \alpha' \mu(Y)$, Inequality (\ref{EQ:expansion beta}) implies $\mu((\r(K_{\alpha'} \cdot B)\setminus B)\cap Y)> C_{\alpha'} \mu(B)$. Then:
\[
\big(\r(K_{\alpha'}\cdot B)\setminus B\big)\cap Y \subseteq \big(\r(K_{\alpha'}\cdot A)\setminus A\big)\cap Y.
\]
Hence
\[
\mu\big((\r(K_{\alpha'}\cdot A)\setminus A)\cap Y\big) \geq \mu\big((\r(K_{\alpha'}\cdot B)\setminus B)\cap Y\big)> C_{\alpha'}\frac{1-\beta}{2\beta} \mu(A).
\]

Finally, set
\[
C\coloneqq \min\left\{C_{\alpha},\frac{1-\beta}{2\beta} C_{\alpha'}, \frac{1-\beta}{2\beta}\right\} \quad \text{and} \quad K\coloneqq K_\alpha \cup K_{\alpha'}.
\]
Then $K$ is unital symmetric and $(N_\alpha+N_{\alpha'})$-decomposable with $\ell(K) \leq \max\{L_{\alpha}, L_{\alpha'}\}$. Hence we conclude the proof.
\end{proof}

\begin{lem}\label{subset also asym expansion1}
Assume $\mathcal{G}$ is asymptotically expanding in measure. Then for any $\beta\in (0,1],\alpha \in (0,\frac{1}{2}]$ and $C\in (0,1)$, there exists a unital symmetric $N$-decomposable $K\subseteq \mathcal{G}$ with $\ell(K) \leq L$ such that for any measurable $Y\subseteq \mathcal{G}^{(0)}$ with $\mu(Y)\geq \beta$ and measurable $A\subseteq Y$ with $\alpha \mu(Y)\leq \mu(A)\leq \frac{1}{2}\mu(Y)$, we have $\mu((\r(K \cdot A)\setminus A)\cap Y)> C \mu(A)$. Here $N,L$ only depend on $\alpha,\beta,C$ and expansion parameters of $\G$.
\end{lem}

\begin{proof}
Fix $\alpha\in (0,\frac{1}{2}],\beta\in (0,1]$, $C\in (0,1)$ and a measurable $Y \subseteq \Gz$ with $\mu(Y)\geq \beta$. Set $\epsilon\coloneqq \frac{1+C}{2}\mu(Y)$ and $\gamma\coloneqq\mu(\mathcal{G}^{(0)}\setminus Y)+\epsilon=1-\frac{1-C}{2}\mu(Y)$. It suffices to find a unital symmetric decomposable $K \subseteq \G$ such that for any measurable $A\subseteq Y$ with $\alpha \mu(Y)\leq \mu(A)\leq \frac{1}{2}\mu(Y)$, then $\mu(\r(K\cdot A))>\gamma$. If this holds, then the proof would be completed by the following estimate:
\[
\mu\big(\r(K\cdot A)\cap Y\big)>\gamma-\mu(\G^{(0)}\setminus Y)=\epsilon\geq \frac{2\epsilon}{\mu(Y)}\mu(A)=(1+C)\mu(A).
\]

Now we aim to find such a $K$. Since $\G$ is asymptotically expanding, Lemma \ref{extend to beta} provides $C',L', N'>0$ and a unital symmetric $N'$-decomposable $K'\subseteq \G$ with $\ell(K') \leq L'$ such that for any measurable $A'\subseteq \G^{(0)}$ with $\alpha\beta\leq \mu(A')\leq 1-\frac{1-C}{2}\beta$, then
\begin{equation}\label{EQ:expansion subspace proof}
\mu\big(\r(K'\cdot A')\big)>(1+C')\mu(A')\geq (1+C')\alpha\beta.
\end{equation}
Set $m$ to be the minimal integer satisfying $(1+C')^m \geq \frac{1}{\alpha\beta}$, and take $K\coloneqq (K')^m$. 
Then Lemma \ref{inverse,union,composition still decomposable} shows that $K$ is unital symmetric $(N')^m$-decomposable and $\ell(K) \leq mL'$.

Now for any measurable $A\subseteq Y$ with $\alpha \mu(Y)\leq \mu(A)\leq \frac{1}{2}\mu(Y)$, we need to show that $\mu(\r(K\cdot A))>\gamma$. If not, then $\mu\big(\r((K')^m \cdot A)\big)\leq \gamma$. This shows that
\[
\alpha\beta\leq \mu\big(\r((K')^i\cdot A)\big)\leq \gamma\leq 1-\frac{1-C}{2}\beta\quad\text{for any}\quad  i=0,1,\cdots m.
\]
Applying Inequality (\ref{EQ:expansion subspace proof}) inductively, we obtain:
\[
\mu\big(\r((K')^m \cdot A)\big)>(1+C')^m \mu(A)\geq (1+C')^m \alpha\beta \geq 1 >\gamma,
\]
which leads to a contradiction.
\end{proof}

\subsection{The structure theorem}\label{ssec:structure thm}

Now we are in the position to introduce the structure theorem for asymptotic expansion on groupoids. The main idea is to approximate by domains of expansions in the following sense:

\begin{defn}\label{def of uniform exhaustion}
	Let $\mathcal{G}$ be a groupoid with a length function $\ell$, and $(\Gz,\RR,\mu)$ be a probability measure space.  We say that a sequence of measurable subsets $\{Y_n\}_{n\in \N}$ in $\mathcal{G}^{(0)}$ \emph{forms a (measured) exhaustion of $\Gz$} if $\lim_{n\to \infty}\mu(Y_n)=1$.
\end{defn}


For technical reasons, we also need to consider:

\begin{defn}\label{defn:RN der}
Let $\mathcal{G}$ be a groupoid with a length function $\ell$, and $(\Gz,\RR,\mu)$ be a probability measure space. Given an admissible bisection $K \subseteq \G$, we set
\begin{equation}\label{EQ:RN der}
\RN(K,x)\coloneqq \frac{\d (\tau_K^{-1})_\ast \mu|_{\r(K)}}{\d \mu|_{\s(K)}}(x)
\end{equation}
the Radon-Nikodym derivative at $x\in \s(K)$, where $\tau_K$ comes from (\ref{EQ:alpha_K}). Here $\mu|_{\s(K)}$ and $\mu|_{\r(K)}$ are the restrictions of $\mu$, and $(\tau_K^{-1})_\ast \mu|_K$ is the push-forward measure defined by $\big((\tau_K^{-1})_\ast \mu|_{\r(K)}\big)(A) = \mu(\tau_K(A))$ for measurable $A \subseteq \s(K)$.
\end{defn}

To simplify the statement of our main result, let us package the derivative information into the notion of domain of expansion:

\begin{defn}\label{defn:domain of expansion with bdd ratio}
Let $\mathcal{G}$ be a groupoid with a length function $\ell$, and $(\Gz,\RR,\mu)$ be a probability measure space. For measurable $Y \subseteq \Gz$, we say that $Y$ is a \emph{domain of expansion with bounded ratio} if there exist $C,N,L>0$ and a unital symmetric $N$-decomposable $K \subseteq \G$ with $\ell(K) \leq L$, together with a unital symmetric decomposition $K= \bigcup_{i=1}^N K_i$ satisfying the following:
\begin{enumerate}
 \item for measurable $A \subseteq Y$ with $0 < \mu(A) \leq \frac{1}{2}\mu(Y)$, then we have
 \[
 \mu\big((\r(K_\alpha \cdot A)\setminus A)\cap Y\big)> C_\alpha \mu(A).
 \]
\item there exists $\theta\geq 1$ such that $\frac{1}{\theta}\leq \RN(K_i,x)\leq \theta$ for any $x\in Y$ and $i \in \{1,2,\ \cdots,N\}$ with $\tau_{K_i}(x) \in Y$. (As a priori to $\tau_{K_i}(x) \in Y$, we have $x\in \s(K_i)$ and hence $\RN(K_i,x)$ makes sense.)
\end{enumerate}
In this case, we say that $Y$ is a \emph{domain of $(C,N,L)$-expansion with ratio bounded by $\theta$}.
\end{defn}

The following is our structure theorem:

\begin{thm}\label{exhaustion of expansion with radon controlled}
Let $\mathcal{G}$ be a groupoid with a length function $\ell$, and $(\Gz,\RR,\mu)$ be a probability measure space. Then the following are equivalent:
	\begin{enumerate}
		\item The groupoid $\mathcal{G}$ is asymptotically expanding in measure.
		\item The unit space $\mathcal{G}^{(0)}$ admits an exhaustion by domains $Y_n$ of $(C_n,N_n,L_n)$-expansion of ratio bounded by $\theta_n$ for $C_n,N_n,L_n >0$ and $\theta_n \geq 1$.
		\item The unit space $\mathcal{G}^{(0)}$ admits an exhaustion by domains of asymptotic expansion.
	\end{enumerate}
\end{thm}

Here the extra requirement on the Radon-Nikodym derivatives in condition (2) will play an important role to prove our main results in Section \ref{sec:Main results}. 

A key tool to prove Theorem \ref{exhaustion of expansion with radon controlled} is to consider maximal F{\o}lner sets as follows.

\begin{defn}\label{Folner set}
Given measurable $Y\subseteq \mathcal{G}^{(0)}$, decomposable $K \subseteq \G$ and $\epsilon > 0$, we say that a measurable subset $F \subseteq Y$ is \emph{$(\epsilon,K)$-F\o lner in $Y$} if $\mu(F)\leq \frac{1}{2}\mu(Y)$ and 
\[
\mu\big((\r(K\cdot F)\setminus F)\cap Y\big)\leq \epsilon\mu(F).
\]
When $Y=\Gz$, we simply say that $F$ is \emph{$(\epsilon,K)$-F\o lner}.
\end{defn}

Now fix measurable $Y\subseteq \mathcal{G}^{(0)}$, decomposable $K \subseteq \G$ and $\epsilon > 0$. Denote the set of all $(\epsilon,K)$-F\o lner sets in $Y$ by $\mathcal{F}_{\epsilon,K}$. Consider the equivalence relation on $\mathcal{F}_{\epsilon,K}$ by setting $F\sim F'$ in $\mathcal{F}_{\epsilon,K}$ if and only if they differ by a null-set. Define a partial order on $\mathcal{F}_{\epsilon,K}/\sim$ by setting $[F]\sqsubseteq [F']$ if $F\subseteq F'$ up to null-sets.

The following lemma is similar to \cite[Lemma 4.2]{asymptoticexpansionandstrongergodicity}:

\begin{lem}\label{lem:maximal Folner set}
The partial ordered set $(\mathcal{F}_{\epsilon,K}/\sim,\sqsubseteq)$ has maximal elements.
\end{lem}

\begin{proof}
Take an arbitrary $F_0\in \mathcal{F}_{\epsilon,K}$ and set $\beta_0\coloneqq \sup\{\mu(F)~|~F\in \mathcal{F}_{\epsilon,K} \text{ and } F_0\subseteq F\}$. Inductively, we choose $F_n\in \mathcal{F}_{\epsilon,K}$ such that $F_{n-1}\subseteq F_n$ and $\mu(F_n)> \beta_{n-1}-\frac{1}{n}$ where 
\[
\beta_n\coloneqq \sup\left\{\mu(F)~|~F\in \mathcal{F}_{\epsilon,K} \text{ and } F_n\subseteq F\right\}.
\]
If $\beta_n=\mu(F_n)$ for some $n$, then $[F_n]$ is maximal. Otherwise, we obtain an increasing sequence $\{F_n\}_{n \in \N}$. Setting $F\coloneqq\bigcup_{n\in \mathbb{N}}F_n$ and similar to the proof of \cite[Lemma 4.2]{asymptoticexpansionandstrongergodicity}, we obtain that $[F]$ is maximal. 
\end{proof}

Maximal F{\o}lner sets are useful to produce domains of expansions, which are crucial to prove Theorem \ref{exhaustion of expansion with radon controlled}. We start with the following:

\begin{lem}\label{lem:expansion and Folner set}
Given measurable $Y\subseteq \mathcal{G}^{(0)}$, decomposable $K \subseteq \G$ and $\epsilon > 0$, let $F_{\epsilon,K}\subseteq Y$ be a maximal $(\epsilon,K)$-F\o lner set in $Y$. Then for any measurable $A\subseteq Y\setminus F_{\epsilon,K}$ with $0< \mu(A)\leq \frac{1}{2}\mu(Y)-\mu(F_{\epsilon,K})$, we have 
\[
\mu\big(((\r(K\cdot A)\setminus A)\cap Y)\setminus F_{\epsilon,K}\big)> \epsilon\mu(A).
\]
\end{lem}

\begin{proof}
Given such an $A\subseteq Y\setminus F_{\epsilon,K}$, we have $\mu(F_{\epsilon,K})< \mu(A\sqcup F_{\epsilon,K})\leq \frac{1}{2}\mu(Y)$ and $\big(\r(K\cdot(A\sqcup F_{\epsilon,K}))\setminus (A\sqcup F_{\epsilon,K})\big)\cap Y$ is contained in 
\[
\big(((\r(K\cdot A)\setminus A)\cap Y)\setminus F_{\epsilon,K}\big) \cup \big((\r(K\cdot F_{\epsilon,K})\setminus F_{\epsilon,K})\cap Y\big).
\]
By the maximality of $F_{\epsilon,K}$, it follows that 
\begin{align*}
\epsilon\mu(A\sqcup F_{\epsilon,K}) &< \mu\big((\r(K\cdot(A\sqcup F_{\epsilon,K}))\setminus (A\sqcup F_{\epsilon,K}))\cap Y\big)\\ &\leq\mu\big(((\r(K\cdot A)\setminus A)\cap Y)\setminus F_{\epsilon,K}\big)+\mu\big((\r(K\cdot F_{\epsilon,K})\setminus F_{\epsilon,K})\cap Y\big).
\end{align*}
Since $\mu\big((\r(K\cdot F_{\epsilon,K})\setminus F_{\epsilon,K})\cap Y\big) \leq \epsilon\mu(F_{\epsilon,K})$, we have
\[
\epsilon\mu(A)<
\mu\big(((\r(K\cdot A)\setminus A)\cap Y)\setminus F_{\epsilon,K}\big).
\]
Then we finish the proof.
\end{proof}

\begin{proof}[Proof of Theorem \ref{exhaustion of expansion with radon controlled}]
	``$(1)\Rightarrow(2)$'': Fix $C\in (0,\frac{1}{2})$ and let $\alpha_n \coloneqq \frac{C}{(4+2C)(n+1)}$ for each $n \in \N$. Since $\G$ is asymptotically expanding, Lemma \ref{subset also asym expansion1} provides a unital symmetric $N_n$-decomposable $K_n$ with $\ell(K_n) \leq L_n$ satisfying the condition therein for $\beta=\frac{1}{2}$, $\alpha=\alpha_n$ and $C$. Take a unital symmetric decomposition $K_n=\bigcup_{i=1}^{N_n} K_{n,i}$ for admissible $K_{n,i}$ and set $\theta_n\coloneqq N_n\cdot (n+1)$. Denote
\[
	Z_n \coloneqq \left\{\tau_{K_{n,i}}(x)~\middle|~x\in \s(K_{n,i}) \text{ and } i=1,\cdots, N_n \text{ such that } \RN(K_{n,i},x) < \frac{1}{\theta_n}\right\}.
\]
Then $\mu(Z_n)< \frac{N_n}{\theta_n}\cdot \mu(\mathcal{G}^{(0)})=\frac{1}{n+1}$, which implies that 
	\[
	\mu(\mathcal{G}^{(0)}\setminus Z_n)\geq \frac{n}{n+1}\geq \frac{1}{2}.
	\]
	
	Let $X_n\coloneqq \mathcal{G}^{(0)}\setminus Z_n$, and take $F_n$ to be a maximal $(C,K_n)$-F\o lner sets in $X_n$, ensured by Lemma \ref{lem:maximal Folner set}. Then we have $\mu(F_n) < \alpha_n\cdot \mu(X_n)$. Setting $Y_n\coloneqq X_n \setminus F_n$, then $\mu(Y_n)> (1-\alpha_n)\mu(X_n)\geq (1-\alpha_n)\cdot \frac{n}{n+1}$.
	
Now we claim that $Y_n$ is a domain of $(\frac{C}{2},N_n, L_n)$-expansion. In fact, we take an arbitrary measurable $A \subseteq Y_n$ with $0< \mu(A)\leq \frac{1}{2}\mu(Y_n)$ and divide into two cases:

\noindent \textbf{Case I.} If $\mu(A)\leq \frac{1}{2}\mu(X_n)-\mu(F_n)$, then by Lemma \ref{lem:expansion and Folner set} we have
	\[
	\mu\big((\r(K_n\cdot A)\setminus A)\cap Y_n\big)> C \mu(A).
	\]
\noindent \textbf{Case II.} If $\mu(A)> \frac{1}{2}\mu(X_n)-\mu(F_n)$, then we have 
	\[
	\frac{1}{2}\mu(X_n)\geq \mu(A) > (\frac{1}{2}-\alpha_n)\mu(X_n)\geq \alpha_n \mu(X_n).
	\]
It follows from the requirement on $K_n$ that $\mu\big((\r(K_n\cdot A)\setminus A)\cap X_n\big)> C\mu(A)$. Since
\[
\mu\big((\r(K_n\cdot A)\setminus A)\cap Y_n\big) \geq \mu\big((\r(K_n\cdot A)\setminus A)\cap X_n\big)-\mu(F_n)> C\mu(A)-\mu(F_n)
\]
and 
\[
\mu(A)> (\frac{1}{2}-\alpha_n) \mu(X_n)\geq (\frac{1}{2}-\frac{C}{4+2C})\mu(X_n)=\frac{1}{C+2}\mu(X_n),
\]
then we have
\[
C\mu(A)-\mu(F_n) > C\mu(A)-\alpha_n \mu(X_n) > C\mu(A)-\alpha_n(C+2)\mu(A) \geq \frac{C}{2}\mu(A).
\]

In conclusion, we showed that $Y_n$ is a domain of $(\frac{C}{2},N_n, L_n)$-expansion with ratio bounded by $\theta_n$ and $\mu(Y_n)> (1-\alpha_n)\cdot \frac{n}{n+1}$.
	
	``$(2)\Rightarrow(3)$'' is trivial.
	
	``$(3)\Rightarrow(1)$'': Take an exhaustion of $\Gz$ by domains $Y_n$ of asymptotic expansion. Assume that $\mathcal{G}$ were not asymptotically expanding. Then there exists $\alpha_0\in(0,\frac{1}{2})$ such that for any $C>0$ and any unital symmetric decomposable $K\subseteq \mathcal{G}$, there exists $A_{C,K} \in \RR$ with $\alpha_0\leq \mu(A_{C,K})\leq\frac{1}{2}$ such that $\mu\big(\r(K\cdot A_{C,K})\setminus A_{C,K}\big) \leq C \mu(A_{C,K})$.
	
Take $n\in \N$ such that $\mu(Y_n)\geq 1-\frac{\alpha_0}{2}$. Then for any measurable $A \subseteq \Gz$ with $\alpha_0 \leq \mu(A) \leq\frac{1}{2}$, direct calculations show that $\frac{\alpha_0}{2} \mu(Y_n)\leq \mu(A\cap Y_n)\leq\frac{1}{2-\alpha_0}\mu(Y_n)$. Then Lemma \ref{extend to beta} provides $\epsilon> 0$ and unital symmetric decomposable $K \subseteq \mathcal{G}$ such that 
	 \[
	 \mu\big((\r(K \cdot (A\cap Y_n))\setminus (A\cap Y_n))\cap Y_n\big)> \epsilon \mu(A\cap Y_n) \geq \epsilon(\mu(A)-\frac{\alpha_0}{2})\geq \frac{\epsilon}{2} \mu(A), 
	 \]
which implies that
\[
\mu(\r(K\cdot A)\setminus A) \geq \mu\big((\r(K \cdot (A\cap Y_n))\setminus A)\cap Y_n\big) > \frac{\epsilon}{2}\mu(A).
\]
This leads to a contradiction.
\end{proof}

\begin{rem}\label{rem:depending on parameters}
From the proof of ``$(1)\Rightarrow(2)$'' above, we know that if $\G$ is asymptotically expanding, then in condition (2) we can take $C_n \equiv \frac{C}{2}$ for any chosen $C\in (0,\frac{1}{2})$, domains $Y_n$ satisfying $\mu(Y_n)> \big(1-\frac{C}{(4+2C)(n+1)}\big)\cdot \frac{n}{n+1}$, $N_n, L_n$ and $\theta_n$ only depend on the expansion parameters of $\G$.
\end{rem}

\subsection{Markov kernels on groupoids}\label{ssec:Markov structure theorem}

Here we construct reversible Markov kernels on groupoids and study the relation between their Cheeger constants and the expansion of groupoids.

Firstly, we aim to construct a Markov kernel for decomposable subsets. 
Let us fix a unital symmetric decomposable subset $K$ together with a unital symmetric decomposition $K = \bigcup_{i=1}^N K_i$.

\begin{defn}\label{defn:Markov kernel for K}
For measurable $Y \subseteq \Gz$ and $x\in Y$, denote 
\begin{equation}\label{EQ:aux func for Markov}
K_{Y,x}:=\{i ~|~ x\in \s(K_i) \text{ and }\tau_{K_{i}}(x)\in Y\} \quad \text{and} \quad \sigma_{Y,K}(x)=\sum_{i \in K_{Y,x}} \RN(K_i,x)^{\frac{1}{2}}.
\end{equation}
The \emph{normalized local Markov kernel} associated to $Y$ and $K$ is the Markov kernel on $Y$ defined as follows:
\begin{equation}\label{EQ:Markov for K}
\Pi_{Y,K}(x,-)=\frac{1}{\sigma_{Y,K}(x)}\sum_{i\in K_{Y,x}} \RN(K_{i},x)^{\frac{1}{2}}\delta_{\tau_{K_i}(x)}(-) \quad \text{for} \quad  x\in Y.
\end{equation}
Here $\delta_y$ is the Dirac delta measure on $y$.
\end{defn}

Since the decomposition is unital, it is clear that $\sigma_{Y,K}>0$ on $Y$ and (\ref{EQ:Markov for K}) makes sense. It is also  routine to check that (\ref{EQ:Markov for K}) is indeed a Markov kernel on $Y$. 
To see that $\Pi_{Y,K}$ is reversible, we consider the measure on $Y$ defined by:
\begin{equation}\label{EQ:reservsing measure}
\d\tilde{\mu}_{Y,K} \coloneqq \sigma_{Y,K}\cdot \d(\mu|_Y).
\end{equation}
Note that both $\Pi_{Y,K}$ and $\tilde{\mu}_{Y,K}$ depend on the decomposition of $K$.

We collect several useful properties of $\Pi_{Y,K}$ in the following. The proof is similar to that of \cite[Proposition 3.10]{li2022markovianroealgebraicapproachasymptotic} since the decomposition $K=\bigcup_{i=1}^N K_i$ is unital and symmetric. Hence we omit the details.

\begin{prop}\label{check reversible Markov kernel}
In the setting above, we have:
\begin{enumerate}
\item The measure $\tilde{\mu}_{Y,K}$ is reversing for the Markov kernel $\Pi_{Y,K}$. 
\item For measurable $A \subseteq Y$, we have $\mu(A)\leq \tilde{\mu}_{Y,K}(A) \leq N\sqrt{{\mu(A)\mu(Y)}}$. Hence $\tilde{\mu}_{Y,K}$ is equivalent to the restriction $\mu|_Y$.
\end{enumerate}
\end{prop}

\begin{defn}\label{Markov expansion}
For a measurable $Y \subseteq \Gz$ and $C,N,L > 0$, we call $Y$ a \emph{domain of Markov $(C,N,L)$-expansion} (or simply, \emph{domain of Markov expansion}) if there exists a unital symmetric $N$-decomposable $K$ together with a unital symmetric decomposition $K=\bigcup_{i=1}^N K_i$ such that $\ell(K) \leq L$ and the associated normalised local Markov kernel $\Pi_{Y,K}$ on $(Y, \tilde{\mu}_{Y,K})$ has Cheeger constant greater than $C$. 

Moreover, if the decomposition for $K$ above has ratio bounded by $\theta \geq 1$, then we say that $Y$ is a \emph{domain of Markov $(C,N,L)$-expansion with ratio bounded by $\theta$}.
\end{defn}

The following lemma relates the domain of Markov expansion with the ordinary expansion. The proof is similar to that of \cite[Lemma 3.14]{li2022markovianroealgebraicapproachasymptotic} and hence omitted.

\begin{lem}\label{expansion and Markov expansion}
For a measurable $Y \subseteq \Gz$, we have:
\begin{enumerate}
 \item If $Y$ is a domain of $(C,N,L)$-expansion with ratio bounded by $\theta$, then $Y$ is a domain of $(\frac{C}{N\theta}, N,L)$-Markov expansion;
 \item If $Y$ is a domain of $(\kappa,N,L)$-Markov expansion with ratio bounded by $\theta$, then $Y$ is a domain of $(\frac{\kappa}{N\sqrt{\theta} + \kappa}, N,L)$-expansion. 
\end{enumerate}
\end{lem}

Combining Theorem \ref{exhaustion of expansion with radon controlled} with Lemma \ref{expansion and Markov expansion}, we reach the following Markovian version of the structure theorem:

\begin{thm}\label{thm:Markov structure theorem}
Let $\mathcal{G}$ be a groupoid with a length function $\ell$, and $(\Gz,\RR,\mu)$ be a probability measure space. Then the following are equivalent:
	\begin{enumerate}
		\item The groupoid $\mathcal{G}$ is asymptotically expanding in measure;
		\item The unit space $\mathcal{G}^{(0)}$ admits an exhaustion by domains $Y_n$ of $(C_n, N_n, L_n)$-Markov expansion with ratio bounded by $\theta_n$ for $C_n,N_n,L_n >0$ and $\theta_n \geq 1$.
	\end{enumerate}
\end{thm}

\begin{rem}\label{rem:dependence 2}
Combining Remark \ref{rem:depending on parameters} and the dependence of parameters in Lemma \ref{expansion and Markov expansion}, we know that if $\G$ is asymptotically expanding, then in Theorem \ref{thm:Markov structure theorem}(2) we can take domains $Y_n$ to  satisfy $\mu(Y_n)> t_n$ for some universal $t_n$ independent of $\G$, $C_n, N_n, L_n$ and $\theta_n$ only depend on expansion parameters of $\G$.
\end{rem}

\section{Dynamical propagation and quasi-locality}\label{Dynamical concepts on groupoids}

In this section, we introduce the notion of dynamical propagation and dynamical quasi-locality for operators on groupoids. Again we always assume that $\G$ is a groupoid with a length function $\ell$, and $\Gz$ is equipped with a measure $\mu$ on some $\sigma$-algebra $\RR$. To include more examples, here we do \emph{not} require $\mu$ to be finite.

\subsection{Basic Notions}\label{ssec:basic notions}

We start with the following key notions:

\begin{defn}\label{finite dynamical propagation}
Let $\mathcal{G}$ be a groupoid with a length function $\ell$, and $(\Gz,\RR,\mu)$ be a measure space. An operator $T \in \B(L^2(\Gz,\mu))$ is said to have \emph{finite dynamical propagation} if there is a unital symmetric $N$-decomposable subset $K\subseteq \G$ with $\ell(K) \leq L$ for some $N,L>0$ such that for any $A,B \in \RR$ with $\mu(\r(K \cdot A)\cap B)=0$, then $\chi_A T\chi_B=0$. 
In this case, we also say that $T$ has \emph{$(N,L)$-dynamical propagation}.

Denote the set of all operators with finite dynamical propagation by $\C_{\dyn}(\G)$, and its norm completion by $\CC^*_{\dyn}(\mathcal{G})$, called the \emph{dynamical Roe algebra of $\G$}.
\end{defn}

Similarly, we consider its quasi-local counterpart:

\begin{defn}\label{dynamical quasi local}
Let $\mathcal{G}$ be a groupoid with a length function $\ell$, and $(\Gz,\RR,\mu)$ be a measure space. An operator $T \in \B(L^2(\Gz,\mu))$ is called \emph{dynamically quasi-local} if for any $\epsilon > 0$, there is a unital symmetric $N_\epsilon$-decomposable subset $K_\epsilon \subseteq \G$ with $\ell(K_\epsilon) \leq L_\epsilon$ for some $N_\epsilon, L_\epsilon>0$ such that for any $A,B \in \RR$ with $\mu(\r(K_\epsilon \cdot A)\cap B)=0$, then $\|\chi_A T\chi_B\| < \epsilon$. The maps $\epsilon \mapsto N_\epsilon, \epsilon \mapsto L_\epsilon$ are called \emph{quasi-local parameters} for $T$.

Denote the set of all dynamically quasi-local operators in $\B(L^2(\Gz,\mu))$ by $\CC^*_{\dyn,\q}(\mathcal{G})$, called the \emph{dynamical quasi-local algebra of $\G$}.
\end{defn}

The following shows that $\CC^*_{\dyn}(\mathcal{G})$ and $\CC^*_{\dyn,\q}(\mathcal{G})$ are indeed $C^*$-algebras:

\begin{lem}\label{check of finite propagation algebra}
Given $\G$, $\ell$ and $\mu$ as above, we have:
\begin{enumerate}
 \item The set $\C_{\dyn}(\G)$ is a $\ast$-algebra and hence, $\CC^*_{\dyn}(\mathcal{G})$ is a $C^*$-algebra;
 \item The set $\CC^*_{\dyn,\q}(\mathcal{G})$ is a $C^*$-algebra.
\end{enumerate}
\end{lem}

\begin{proof}
We only prove (1), since (2) is similar.

By Lemma \ref{inverse,union,composition still decomposable}, $\C_{\dyn}(\G)$ forms an algebra. To see it preserves the $\ast$-operator, fix $T \in \C_{\dyn}(\G)$. Then there exists a unital symmetric decomposable subset $K = \bigcup_{i=1}^N K_i$  such that for any $A,B \in \RR$ with $\mu(\r(K \cdot A)\cap B)=0$, then $\chi_A T\chi_B=0$. 
Then we have $\chi_B T^* \chi_A=0$ for any measurable $A,B \subseteq \Gz$ with $\mu(\r(K \cdot A)\cap B)=0$. Since each $\tau_{K_i}$ is measure-class-preserving and $K$ is symmetric, this is equivalent to that $\mu(\r(K \cdot B)\cap A)=0$. Hence we conclude that $T^* \in \C_{\dyn}(\G)$.
\end{proof}

For a sequence of groupoids, we can combine them into a single groupoid and translate the notion of dynamical propagation and quasi-locality thereon in a uniform version. This is important to recover the case of uniform Roe algebras.

More precisely, for each $n\in \N$, let $\G_n$ be a groupoid with a length function $\ell_n$ and $(\Gz_n,\RR_n,\mu_n)$ be a measure space. Form a groupoid $\G$ to be their disjoint union $\bigsqcup_{n\in \N} \G_n$, and equip $\Gz=\bigsqcup_{n\in \N} \Gz_n$ with a measure $\mu$ on some $\sigma$-algebra $\RR$ generated by $\bigsqcup_{n\in \N} \RR_n$ and $\mu$ is determined by $\mu|_{\Gz_n} \coloneqq \mu_n$ for each $n$. Also take a length function $\ell$ on $\G$ to be the disjoint union of $\ell_n$. Then we have
\[
L^2(\Gz, \mu) = \bigoplus_{n\in \N} L^2(\Gz_n, \mu_n).
\]
The following shows that operators in $\CC^*_{\dyn,\q}(\G)$ are always diagonal and can be described in a uniform version:

\begin{lem}\label{family of quasi local and quasi local}
Given $T \in \B(L^2(\Gz,\mu))$, we have:
\begin{enumerate}
 \item $T \in \C_{\dyn}(\G)$ \emph{if and only if} there exists $T_n \in \B(L^2(\Gz_n, \mu_n))$ for each $n\in \N$ with $\sup_{n\in \N} \|T_n\| < \infty$ such that $T= (\SOT)-\sum_{n \in \N} T_n$ and there exist $N,L>0$ satisfying: for any $n\in \N$ there exists a unital symmetric $N$-decomposable $K_n \subseteq \G_n$ with $\ell_n(K_n) \leq L$ such that for any $A,B \in \RR_n$ with $\mu(\r(K_n \cdot A)\cap B)=0$, we have $\chi_A T_n \chi_B=0$. 
 \item $T \in \CC^*_{\dyn,\q}(\G)$ \emph{if and only if} there exists $T_n \in \B(L^2(\Gz_n, \mu_n))$ for each $n\in \N$ with $\sup_{n\in \N} \|T_n\| < \infty$ such that $T= (\SOT)-\sum_{n \in \N} T_n$ and for any $\epsilon > 0$ there exist $N, L>0$ satisfying: for any $n\in \N$, there exists a unital symmetric $N$-decomposable $K_n \subseteq \G_n$ with $\ell_n(K_n) \leq L$ such that for any $A,B \in \RR_n$ with $\mu(\r(K_n \cdot A)\cap B)=0$, we have $\|\chi_A T_n\chi_B\| < \epsilon$. 
\end{enumerate}
\end{lem}

\begin{proof}
(1). \emph{Necessity}: Firstly, we show that $T$ is diagonal. Note that for any $n\neq m$ and $A \in \RR_n$, $B \in \RR_m$, we have $(K \cdot A) \cap B = \emptyset$ for any decomposable $K \subseteq \G$. Then it follows from definition that $\chi_A T \chi_B = 0$, and hence $T= (\SOT)-\sum_{n \in \N} T_n$ for $T_n=\chi_{\Gz_n} T \chi_{\Gz_n}$. 

Furthermore, there exists a unital symmetric decomposable $K \subseteq \G$ such that for any $A,B \in \RR$ with $\mu(\r(K \cdot A)\cap B)=0$, then $\chi_A T\chi_B=0$. Assume that $K$ is $N$-decomposable and $\ell(K) \leq L$ for some $N,L>0$. For each $n\in \N$, set $K_n \coloneqq K \cap \G_n$. Then $K_n$ is $N$-decomposable, unital and symmetric with $\ell_n(K_n) \leq L$. It is easy to check that $K_n$ satisfies the requirement.

\emph{Sufficiency}: Assume $K_n$ satisfies the requirement, and write $K_n=\bigcup_{i=1}^N K_n^{(i)}$ for admissible $K_n^{(i)} \subseteq \G_n$ with $\ell_n(K_n^{(i)}) \leq L$. Define $K\coloneqq \bigsqcup_{n\in \N} K_n \subseteq \G$ and for $i=1,\cdots, N$, define $K^{(i)} \coloneqq \bigsqcup_{n\in \N} K_n^{(i)}$. Since $K_n^{(i)}$ is admissible, it is easy to see that $K^{(i)} \subseteq \G$ is admissible for each $i$. Hence $K$ is $N$-decomposable and $\ell(K) \leq L$. 

Now for any $A,B \in \RR$ with $\mu(\r(K \cdot A)\cap B)=0$, set $A_n\coloneqq A \cap \Gz_n$ and $B_n \coloneqq B \cap \Gz_n$ and then we have $\mu_n(\r(K_n \cdot A_n)\cap B_n)=0$ for each $n\in \N$. By assumption, we have $\chi_{A_n} T_n \chi_{B_n} = 0$. Finally, note that
\[
\chi_A T \chi_B = \sum_{n\in \N} \chi_{A_n} T_n \chi_{B_n} = 0,
\]
which concludes (1). (2) is similar to (1) and hence omitted.
\end{proof}

\subsection{Quasi-locality of the averaging projection}

Now we focus on a special projection operator in $\B(L^2(\Gz,\mu))$ when $\mu$ is a probability measure, whose dynamical quasi-locality is closely related to the asymptotic expansion of the groupoid. 

\begin{defn}\label{defn:ave proj}
Let $\mathcal{G}$ be a groupoid with a length function $\ell$, and $(\Gz,\RR,\mu)$ be a measure space. For any measurable $Y \subseteq \Gz$ with $0< \mu(Y) < \infty$, denote by $P_Y \in \B(L^2(\Gz, \mu))$ the \emph{averaging projection on $Y$}, which is the orthogonal projection onto the one-dimensional subspace in $L^2(\Gz,\mu)$ spanned by $\chi_Y$. In other words, 
\[
P_Y(f)\coloneqq \frac{\langle f,\chi_Y\rangle}{\mu(Y)}\chi_Y \quad \text{for} \quad f \in L^2(\Gz,\mu).
\]
If $\mu(\Gz) < \infty$, we simply write $P_{\G}$ for $P_{\Gz}$.
\end{defn}

To study the quasi-locality of the averaging projection, we firstly record the following. The proof is similar to that of \cite[Lemma 3.8]{Li2019QuasilocalAA}, and hence omitted.

\begin{lem}\label{well known fact for averaging projection}
Assume that $\mu(\Gz) =1$. Then for any measurable $A,B \subseteq \Gz$, we have
\[
\|\chi_A P_\G \chi_B\| = \sqrt{\mu(A) \mu(B)}.
\]
\end{lem}

The following relates the quasi-locality of $P_{\G}$ with asymptotic expansion of $\G$:

\begin{prop}\label{asym expansion and quasi local} 
Let $\mathcal{G}$ be a groupoid with a length function $\ell$, and $(\Gz,\RR,\mu)$ be a probability measure space. Then the averaging projection $P_{\G}$ is dynamically quasi-local \emph{if and only if} $\G$ is asymptotically expanding in measure $\mu$.
\end{prop}


\begin{proof} 
\emph{Necessity}. 
Assume $\epsilon \mapsto N_\epsilon$, $\epsilon \mapsto L_\epsilon$ are quasi-local parameters for $P_\G$. Given $\alpha \in (0,\frac{1}{2}]$, set $N\coloneqq N_{\sqrt{\alpha}/2}$ and $L\coloneqq L_{\sqrt{\alpha}/2}$. Then there exists a unital symmetric $N$-decomposable $K$ with $\ell(K) \leq L$ such that for any $A,B \in \RR$ with $\mu(\r(K \cdot A)\cap B)=0$, we have $\|\chi_A T\chi_B\| < \frac{\sqrt{\alpha}}{2}$. Then for any $A \in \RR$ with $\alpha \leq \mu(A) \leq \frac{1}{2}$, Lemma \ref{well known fact for averaging projection} shows:
\[
\sqrt{\mu(A) \cdot \mu\big(\Gz \setminus \r(K \cdot A)\big)} = \|\chi_A P_\G \chi_{\Gz \setminus \r(K \cdot A)}\| <  \frac{\sqrt{\alpha}}{2}.
\]
Hence we obtain 
\[
\mu(A) \cdot \mu\big(\Gz \setminus \r(K \cdot A)\big) < \frac{\alpha}{4},
\]
which implies that $\mu\big(\Gz \setminus \r(K \cdot A)\big) < \frac{1}{4}$. Therefore, we obtain:
\[
\mu(\r(K \cdot A)) > \frac{3}{4} \geq (1+\frac{1}{2}) \cdot \mu(A).
\]
This concludes that $\G$ is asymptotically expanding in measure.

\emph{Sufficiency}. 
Assuming $\alpha \mapsto C_\alpha$, $\alpha \mapsto N_\alpha$ and $\alpha \mapsto L_\alpha$ are expansion parameters for $\G$, we take $K_\alpha$ as in Definition \ref{asymptotically expanding}. Given $0<\epsilon\leq \frac{1}{2}$, take $n \in \N$ to be the smallest number such that 
\[
(1+C_\epsilon)^n \cdot \epsilon > \frac{1}{2}.
\]
Set $K\coloneqq K_\epsilon^{2n}$, which is unital symmetric $N_\epsilon^{2n}$-decomposable with length at most $2nL_\epsilon$ by Lemma \ref{inverse,union,composition still decomposable}. 

For any measurable $A, B \subseteq \Gz$ with $\mu(\r(K \cdot A) \cap B) = 0$, we have
\[
\mu\big(\r(K_\epsilon^n \cdot A) \cap \r(K_\epsilon^n \cdot B) \big) = 0
\]
since $K_\epsilon$ is symmetric. Hence without loss of generality, we assume
\[
\mu(\r(K_\epsilon^n \cdot A)) \leq \frac{1}{2}.
\]
If $\mu(A) < \epsilon$, then by Lemma \ref{well known fact for averaging projection} we have 
\[
\|\chi_A P_\G \chi_B\| = \sqrt{\mu(A)\mu(B)} < \sqrt{\epsilon}.
\]
If $\mu(A) \geq \epsilon$, then using asymptotic expansion inductively, we have:
\[
\frac{1}{2} \geq \mu(\r(K_\epsilon^n \cdot A)) \geq (1+C_\epsilon) \mu(\r(K_\epsilon^{n-1} \cdot A)) \geq \cdots \geq (1+C_\epsilon)^n \mu(A) \geq (1+C_\epsilon)^n \cdot \epsilon.
\]
This leads to a contradiction to the choice of $n$. 
\end{proof}

\begin{rem}\label{rem:dependence of parameters 2}
From the proof above, it is clear that if $P_\G$ is dynamically quasi-local, then we can choose expansion parameters for $\G$ only depending on quasi-local parameters for $P_\G$, and vice versa.
\end{rem}

\section{Main results}\label{sec:Main results}

Having established Proposition \ref{asym expansion and quasi local}, now we are in the position to prove the following fundamental case of the main result:


\begin{thm}\label{thm:main thm ave proj}
Let $\mathcal{G}$ be a groupoid with a length function $\ell$, and $(\Gz,\RR,\mu)$ be a probability measure space. Then the following are equivalent:
\begin{enumerate}
 \item the averaging projection $P_\G \in \CC^*_{\dyn}(\mathcal{G})$;
 \item the averaging projection $P_\G \in \CC^*_{\dyn,\q}(\mathcal{G})$;
 \item the groupoid $\G$ is asymptotically expanding in measure.
\end{enumerate}
\end{thm}

\begin{proof}
``(1) $\Rightarrow$ (2)'' is trivial, and ``(2) $\Leftrightarrow$ (3)'' comes from Proposition \ref{asym expansion and quasi local}. Hence we only focus on ``(3) $\Rightarrow$ (1)''. 

Fix $C\in (0,\frac{1}{2})$. By Theorem \ref{thm:Markov structure theorem}, Remark \ref{rem:depending on parameters} and Remark \ref{rem:dependence 2}, there exists a sequence of measurable subsets $Y_n \subseteq \Gz$ with $\mu(Y_n)> \big(1-\frac{C}{(4+2C)(n+1)}\big)\cdot \frac{n}{n+1}$ such that each $Y_n$ is a domain of $(C_n, N_n, L_n)$-Markov expansion with ratio bounded by $\theta_n$ for $C_n,N_n,L_n >0$ and $\theta_n \geq 1$ only depending on expansion parameters of $\G$. From Definition \ref{Markov expansion}, we can choose a unital symmetric $N_n$-decomposable $K_n$ together with a unital symmetric decomposition $K_n = \bigcup_{i=1}^{N_n} K_n^{(i)}$ such that $\ell(K_n) \leq L_n$ and the associated normalised local Markov kernel $\Pi_n \coloneqq \Pi_{Y_n,K_n}$ from (\ref{EQ:Markov for K}) on $(Y_n, \tilde{\mu}_n)$ has Cheeger constant greater than $C_n$, where $\tilde{\mu}_n \coloneqq \tilde{\mu}_{Y_n,K_n}$ is the reversing measure defined in (\ref{EQ:reservsing measure}). We have a function $\sigma_n\coloneqq \sigma_{Y_n, K_n}$ from (\ref{EQ:aux func for Markov}) such that $\sigma_n\geq 1$ on $Y_n$. 
Denote the associated Markov operator by $\P_n$ from (\ref{EQ:Markov operator defn}) with spectral gap $\lambda_n$.

Now for each $n\in \N$, we consider the following embedding
\[
I_n: L^2(Y_n,\tilde{\mu}_n) \longrightarrow L^2(\Gz,\mu)
\]
simply by extending each function in $L^2(Y_n,\tilde{\mu}_n)$ by zero on $\Gz \setminus Y_n$. This is well-defined since $\tilde{\mu}_n$ is equivalent to $\mu|_{Y_n}$ and we have $\|I_n\| \leq 1$ thanks to Proposition \ref{check reversible Markov kernel}(2). Direct calculations show that its adjoint has the form:
\[
I_n^*: L^2(\Gz,\mu) \longrightarrow L^2(Y_n,\tilde{\mu}_n), \quad g \mapsto \frac{1}{\sigma_n} \cdot g|_{Y_n} \quad \text{for} \quad g \in L^2(\Gz,\mu).
\]
Denote the adjoint map
\[
\Ad_n: \B(L^2(Y_n,\tilde{\mu}_n)) \longrightarrow \B(L^2(\Gz,\mu)), \quad T \mapsto I_n \circ T \circ I_n^* \quad \text{for} \quad T \in \B(L^2(Y_n,\tilde{\mu}_n)).
\]

Recall from Definition \ref{defn:ave proj} that we have the averaging projection $P_n \coloneqq P_{Y_n}$ in $\B(L^2(\Gz,\mu))$, which is the orthogonal projection onto constant functions on $Y_n$ in $L^2(\Gz,\mu)$. Denote another orthogonal projection $\tilde{P}_n \in \B(L^2(Y_n,\tilde{\mu}_n))$ onto constant functions on $Y_n$ in $L^2(Y_n,\tilde{\mu}_n)$. Direct calculations show that
\begin{equation}\label{EQ:relation between P and tildeP}
\Ad_n(\tilde{P}_n) = \frac{\mu(Y_n)}{\tilde{\mu}_n(Y_n)} \cdot P_n.
\end{equation}

On the other hand, Theorem \ref{thm:spectral.characterisation.markov.exp} shows that $\frac{C_n^2}{2}\leq 1-\lambda_n\leq 2C_n$. Hence $\frac{1}{2} \chi_{Y_n} + \frac{1}{2}\P_n$ has spectrum contained in $[-\frac{3}{4},1-\frac{C_n^2}{4}]\cup\{1\}$. Therefore, for any $m \in \N$ we have
\[
\left\|\left(\frac{1}{2}\chi_{Y_n}+\frac{1}{2}\P_n\right)^m-\tilde{P}_n\right\| \leq \left(1-\frac{C_n^2}{4}\right)^m.
\]
Applying $\Ad_n$ and using (\ref{EQ:relation between P and tildeP}), we obtain:
\[
\left\|\frac{\tilde{\mu}_n(Y_n)}{\mu(Y_n)} \cdot\Ad_n\left[\left(\frac{1}{2}\chi_{Y_n}+\frac{1}{2}\P_n\right)^m\right]- P_n\right\| \leq \frac{\tilde{\mu}_n(Y_n)}{\mu(Y_n)} \cdot \left(1-\frac{C_n^2}{4}\right)^m.
\]
Note that 
\[
\frac{\tilde{\mu}_n(Y_n)}{\mu(Y_n)} \leq \|\sigma_n\|_\infty \leq \left\| \sum_{i=1}^{N_n} \RN(K_n^{(i)},x)^{\frac{1}{2}} \right\|_\infty \leq N_n\sqrt{\theta_n}.
\]
Combining the above together, we obtain
\begin{equation}\label{EQ: app for proj}
\left\|\frac{\tilde{\mu}_n(Y_n)}{\mu(Y_n)} \cdot\Ad_n\left[\left(\frac{1}{2}\chi_{Y_n}+\frac{1}{2}\P_n\right)^m\right]- P_n\right\| \leq N_n\sqrt{\theta_n} \cdot \left(1-\frac{C_n^2}{4}\right)^m.
\end{equation}
Hence given $\epsilon>0$, we can choose $m_n \in \N$ such that 
\begin{equation}\label{EQ:estimate in pf of main thm}
N_n\sqrt{\theta_n} \cdot \left(1-\frac{C_n^2}{4}\right)^{m_n} < \frac{\epsilon}{2}. 
\end{equation}
Moreover, direct calculations show that for each $n\in \N$ we have:
\[
\|P_n - P_\G\| \leq \sqrt{\mu(\Gz \setminus Y_n)}  = \sqrt{\frac{1}{n+1} + \frac{nC}{(4+2C)(n+1)^2}}.
\]
For $\epsilon$ above, we can further choose $\tilde{N} \in \N$ such that for any $n >  \tilde{N}$ we have 
\begin{equation}\label{EQ:estimate in pf of main thm2}
\sqrt{\frac{1}{n+1} + \frac{nC}{(4+2C)(n+1)^2}} < \frac{\epsilon}{2}.
\end{equation}
Combining with (\ref{EQ: app for proj}), (\ref{EQ:estimate in pf of main thm}) and (\ref{EQ:estimate in pf of main thm2}), we obtain that for any $n >  \tilde{N}$ we have:
\begin{small}
\begin{align}\label{EQ:estimate in pf of main thm3}
\Big\|\frac{\tilde{\mu}_n(Y_n)}{\mu(Y_n)}  \cdot\Ad_n&\left[\left(\frac{1}{2}\chi_{Y_n}+\frac{1}{2}\P_n\right)^{m_n}\right]- P_\G\Big\| < \Big\|\frac{\tilde{\mu}_n(Y_n)}{\mu(Y_n)} \cdot\Ad_n\left[\left(\frac{1}{2}\chi_{Y_n}+\frac{1}{2}\P_n\right)^{m_n}\right]- P_n\Big\| + \|P_n - P_\G\| \nonumber \\
&<N_n\sqrt{\theta_n} \cdot \left(1-\frac{C_n^2}{4}\right)^{m_n} + \sqrt{\frac{1}{n+1} + \frac{nC}{(4+2C)(n+1)^2}} < \frac{\epsilon}{2} + \frac{\epsilon}{2} = \epsilon.
\end{align}
\end{small}

Finally, for $n\in \N$ and measurable $A,B \subseteq Y_n$ with $\mu((K_n \cdot A) \cap B) = 0$, we have $\tilde{\mu}_n((K_n \cdot A) \cap B) = 0$ by Proposition \ref{check reversible Markov kernel}(2). For $x\in Y_n$ and $\xi \in L^2(Y_n, \tilde{\mu}_n)$, we have:
\begin{align*}
(\chi_A \P_n \chi_B \xi)(x) &= \chi_A(x) \cdot \int_{B} \xi(y) \Pi_n(x, \dy) \\
&= \chi_A(x) \cdot \frac{1}{\sigma_n(x)}\sum_{i:\tau_{K_n^{(i)}}(x) \in B} \RN(K_n^{(i)},x)^{\frac{1}{2}} \xi(\tau_{K_n^{(i)}}(x)).
\end{align*}
It follows that $\chi_A \P_n \chi_B=0$. Note that the operator $I_n$ does not change the propagation and hence, the following  operator
\[
\frac{\tilde{\mu}_n(Y_n)}{\mu(Y_n)}  \cdot\Ad_n\left[\left(\frac{1}{2}\chi_{Y_n}+\frac{1}{2}\P_n\right)^{m_n}\right] \quad \text{in} \quad \B(L^2(\Gz,\mu))
\]
has $(N_n^{m_n}, m_n \cdot L_n)$-propagation.
Combining with (\ref{EQ:estimate in pf of main thm3}), we conclude the proof.
\end{proof}

\begin{rem}\label{rem:dependence on parameters for main thm}
From the proof above together with Remark \ref{rem:depending on parameters} and Remark \ref{rem:dependence 2}, we know that if $\G$ is asymptotically expanding in measure, then for any $\epsilon>0$ we can choose $T_\epsilon \in \C_{\dyn}(\G)$ with $(N_\epsilon, L_\epsilon)$-dynamical propagation such that $\|T_\epsilon - P_\G\| < \epsilon$ and the functions $\epsilon \mapsto N_\epsilon$ and $\epsilon \mapsto L_\epsilon$ only depend on expansion parameters of $\G$.
\end{rem}

In the following, we consider general rank-one projections on $L^2(\Gz,\mu)$ for general $\mu$. Let $P \in \B(L^2(\Gz,\mu))$ be a rank-one projection and $\xi \in L^2(\Gz,\mu)$ be a unit vector in the range of $P$. Then
\[
P(\eta)=\langle \eta,\xi\rangle \xi \quad \text{for} \quad \eta \in L^2(\Gz,\mu).
\]
This induces a probability measure $\nu$ on $(\Gz,\RR)$ defined by
\[
\d \nu(x) \coloneqq |\xi(x)|^2 \d \mu(x) \quad \text{for} \quad x\in \Gz.
\]
It is clear that the measure $\nu$ only depends on $P$, called the \emph{associated measure to $P$}. 

Then we have the following:

\begin{thm}\label{thm:main thm ave proj general}
Let $\mathcal{G}$ be a groupoid with a length function $\ell$, and $(\Gz,\RR,\mu)$ be a (not necessarily finite) measure space. Let $P \in \B(L^2(\Gz,\mu))$ be a rank-one projection, and $\nu$ the associated probability measure on $\Gz$. Then the following are equivalent:
\begin{enumerate}
 \item $P \in \CC^*_{\dyn}(\mathcal{G})$;
 \item $P \in \CC^*_{\dyn,\q}(\mathcal{G})$;
 \item $\G$ is asymptotically expanding in measure $\nu$.
\end{enumerate}
\end{thm}


Our idea is to apply Theorem \ref{thm:main thm ave proj} to the probability measure $\nu$. Firstly, note that $\nu$ might \emph{not} be equivalent to $\mu$ and hence, we consider
\[
Z\coloneqq \{x\in \mathcal{G}^{(0)} ~|~\xi(x)=0\}.
\]
Here $Z$ is well-defined up to $\mu$-null sets. We set $Y \coloneqq \Gz \setminus Z$ and then the restrictions $\mu|_Y$ and $\nu|_Y$ are equivalent. Now we consider the reduction $\G_Y^Y=\s^{-1}(Y) \cap \r^{-1}(Y)$,
equipped with the restriction $\ell_Y$ of the length function $\ell$. Denote
\[
Q: L^2(Y,\mu|_Y) \longrightarrow L^2(\Gz, \mu)
\]
the embedding by extending functions in $L^2(Y,\mu|_Y)$ to $0$ on $Z$. Then 
we have:

\begin{lem}\label{lem:reduction Q}
With the notation above, we have:
\begin{enumerate}
 \item $P \in \CC^*_{\dyn,\q}(\mathcal{G})$ \emph{if and only if} $Q^*PQ \in \CC^*_{\dyn,\q}(\G_Y^Y)$;
  \item $P \in \CC^*_{\dyn}(\mathcal{G})$ \emph{if and only if} $Q^*PQ \in \CC^*_{\dyn}(\G_Y^Y)$.
\end{enumerate}
Here we equip $\G_Y^Y$ with the length function $\ell_Y$ and $(\G_Y^Y)^{(0)}=Y$ with the measure $\mu|_Y$ to define $\CC^*_{\dyn}(\G_Y^Y)$ and $\CC^*_{\dyn,\q}(\G_Y^Y)$.
\end{lem}

\begin{proof}
Here we only prove (1), and (2) is similar.

\emph{Necessity}: By assumption, for any $\epsilon>0$ there exists a unital symmetric decomposable subset $K \subseteq \G$ such that for any measurable $A, B \subseteq \Gz$ with $\mu(\r(K \cdot A) \cap B)=0$, we have $\|\chi_A P \chi_B\| < \epsilon$. Set $K'\coloneqq K \cap \G_Y^Y$, which is clearly unital symmetric and decomposable in $\G_Y^Y$. For any measurable $A',B' \subseteq Y$ with $\mu(\r(K' \cdot A') \cap B')=0$, it is easy to see that $\mu(\r(K \cdot A') \cap B')=0$. Hence 
\[
\|\chi_{A'} Q^*PQ \chi_{B'}\| = \|Q^* \chi_{A'} P \chi_{B'}Q\| \leq \|\chi_{A'} P \chi_{B'}\| < \epsilon.
\]

\emph{Sufficiency}: By assumption, for any $\epsilon>0$ there exists a unital symmetric decomposable subset $K \subseteq \G_Y^Y$ such that for any measurable subsets $A, B \subseteq Y$ with $\mu(\r(K \cdot A) \cap B)=0$, we have $\|\chi_A Q^*PQ \chi_B\| < \epsilon$. Set $\hat{K}\coloneqq K \cup \Gz$, which is clearly unital symmetric and decomposable in $\G$. By the construction of $Q$, we have
\[
P = Q(Q^*PQ)Q^*.
\]
Hence for any measurable $\hat{A}, \hat{B} \subseteq \Gz$ with $\mu(\r(\hat{K} \cdot \hat{A}) \cap \hat{B})=0$, we have
\[
\|\chi_{\hat{A}} P \chi_{\hat{B}}\| = \|\chi_{\hat{A}} Q(Q^*PQ)Q^* \chi_{\hat{B}}\| = \|Q\chi_{\hat{A} \cap Y} (Q^*PQ) \chi_{\hat{B} \cap Y}Q^*\| \leq \|\chi_{\hat{A} \cap Y} (Q^*PQ) \chi_{\hat{B} \cap Y}\|  < \epsilon.
\]
Here we use the fact that $\mu\big(\r(K \cdot (\hat{A}\cap Y)) \cap (\hat{B}\cap Y)\big)=0$.
\end{proof}

Consequently, to prove Theorem \ref{thm:main thm ave proj general} we only need consider the case that $\mu$ and $\nu$ are equivalent, \emph{i.e.}, $\mu(Z) = 0$. Therefore in the following, we assume that $Z=\emptyset$.

Now we form the associated dynamical Roe and quasi-local algebras for $\G$ using the measures $\mu$ and $\nu$ on $\Gz$. To tell the difference, we denote $\CC^*_{\dyn}(\G;\mu)$ and $\CC^*_{\dyn,\q}(\G;\mu)$ for the algebras defined using $\mu$, while $\CC^*_{\dyn}(\G;\nu)$ and $\CC^*_{\dyn,\q}(\G;\nu)$ for those using $\nu$. 
Then we have the following:

\begin{lem}\label{lem:equiv measure quasi-local and ppg}
With the notation above and assuming $Z=\emptyset$, we have:
\begin{enumerate}
 \item $P \in \CC^*_{\dyn,\q}(\G;\mu)$ \emph{if and only if} $P_\G \in \CC^*_{\dyn,\q}(\G;\nu)$;
 \item $P \in \CC^*_{\dyn}(\G;\mu)$ \emph{if and only if} $P_\G \in \CC^*_{\dyn}(\G;\nu)$.
\end{enumerate}
Here $P_\G$ is the averaging projection in $\B^2(L^2(\Gz,\nu))$ from Definition \ref{defn:ave proj}.
\end{lem}

\begin{proof}
We only prove (1), since (2) is similar. Denote
\[
U: L^2(\Gz,\nu) \longrightarrow L^2(\Gz,\mu) ,\quad f \mapsto f\cdot \xi \quad \text{for} \quad f \in L^2(\Gz,\nu).
\]
By the construction and the assumption that $Z = \emptyset$ (which implies that $\xi$ is non-zero everywhere), it is easy to check that $U$ is a unitary operator with $U^*\eta = \frac{1}{\xi} \cdot \eta$ for $\eta \in L^2(\Gz,\mu)$.
Denote
\[
\Ad_U: \B(L^2(\Gz,\nu)) \longrightarrow \B(L^2(\Gz,\mu)), \quad T \mapsto UTU^* \quad \text{for} \quad T \in \B(L^2(\Gz,\nu)).
\]
We claim: $\Ad_U(P_\G) = P$. In fact, given $\eta \in L^2(\Gz,\mu)$, we have
\begin{align*}
\Ad_U(P_\G)\eta &= U P_\G U^* \eta = U P_\G \left(\frac{1}{\xi} \cdot \eta\right) =U\left(  \big\langle \frac{1}{\xi} \cdot \eta,\chi_{\Gz}\big\rangle_{L^2(\Gz, \nu)} \chi_{\Gz} \right) \\
&= \big\langle \frac{1}{\xi} \cdot \eta,\chi_{\Gz}\big\rangle_{L^2(\Gz, \nu)} \xi = \left(\int_{\Gz} \frac{1}{\xi}(x) \eta(x) \d \nu(x) \right) \cdot \xi\\
&= \left(\int_{\Gz} \eta(x)\xi(x) \d \mu(x) \right) \cdot \xi = \langle \eta, \xi \rangle_{L^2(\Gz, \mu)} \xi= P\eta.
\end{align*}
Hence we prove the claim. Therefore, for any measurable $A,B \subseteq \Gz$ we have
\[
\Ad_U(\chi_A P_\G \chi_B) = \Ad_U(\chi_A) \circ \Ad_U(P_\G) \circ \Ad_U(\chi_B) = \chi_A P \chi_B.
\]

Finally, since $\mu$ and $\nu$ are equivalent, we have $K \subseteq \G$ is decomposable with respect to $\mu$ if and only if it is decomposable with respect to $\nu$. Moreover, for any measurable $A,B \subseteq \Gz$ and unital symmetric decomposable $K \subseteq \G$, we have $\mu(\r(K \cdot A) \cap B) = 0$ if and only if $\nu(\r(K \cdot A) \cap B) = 0$. So we finish the proof.
\end{proof}

Now we are in the position to prove Theorem \ref{thm:main thm ave proj general}:

\begin{proof}[Proof of Theorem \ref{thm:main thm ave proj general}]
By Lemma \ref{lem:reduction Q}, we can assume that $Z = \emptyset$. Assume $P$ is dynamically quasi-local. It follows from Lemma \ref{lem:equiv measure quasi-local and ppg} that $P_{\G}\in  \CC^*_{\dyn,\q}(\G;\nu)$. Applying Theorem \ref{thm:main thm ave proj}, we obtain that $P_{\G}\in  \CC^*_{\dyn}(\G;\nu)$. Then using Lemma \ref{lem:equiv measure quasi-local and ppg} again, we obtain that $P \in \CC^*_{\dyn}(\G;\mu) = \CC^*_{\dyn}(\G)$.
\end{proof}

\begin{rem}\label{rem:dependence on parameters for general rank one proj}
Note that the map $\Ad_U$ in the proof of Lemma \ref{lem:equiv measure quasi-local and ppg} preserves propagation. Hence combining with Remark \ref{rem:dependence of parameters 2} and \ref{rem:dependence on parameters for main thm}, we know that if the rank-one projection $P$ belongs to $\CC^*_{\dyn,\q}(\G)$, then for any $\epsilon>0$ we can choose $T_\epsilon \in \C_{\dyn}(\G)$ with $(N_\epsilon, L_\epsilon)$-dynamical propagation such that $\|T_\epsilon - P\| < \epsilon$ and the functions $\epsilon \mapsto N_\epsilon$ and $\epsilon \mapsto L_\epsilon$ only depend on quasi-local parameters of $P$.
\end{rem}

Finally, we introduce two variants of Theorem \ref{thm:main thm ave proj general}.
The first one deals with a family version, following the discussion in Section \ref{ssec:basic notions}. Combining Lemma \ref{family of quasi local and quasi local} and Theorem \ref{thm:main thm ave proj general} together with Remark \ref{rem:dependence on parameters for general rank one proj}, we obtain:


\renewcommand{\thethmprime}{5.3'}
\begin{thmprime}\label{cor:cor to thm general case prime} 
For each $n\in \N$, let $\G_n$ be a groupoid with a length function $\ell_n$ and $(\Gz_n,\RR_n,\mu_n)$ be a measure space. Form the groupoid $\G=\bigsqcup_{n\in \N} \G_n$ together with a length function $\ell$ and a measure $\mu$ on $\Gz$ as in Section \ref{ssec:basic notions}. Moreover, let $P_n \in \B(L^2(\Gz_n, \mu_n))$ be a rank-one orthogonal projection for each $n \in \N$ and consider their direct sum
\[
P\coloneqq \mathrm{(SOT)}-\sum_{n\in \N} P_n \in \B(L^2(\Gz,\mu)).
\]
Then 
the following are equivalent:
\begin{enumerate}
 \item $P \in \CC^*_{\dyn}(\mathcal{G})$;
 \item $P \in \CC^*_{\dyn,\q}(\mathcal{G})$;
 \item $\G_n$ is asymptotically expanding in the associated measure $\nu_n$ to $P_n$ uniformly in the sense that they have same expansion parameters.
\end{enumerate}
\end{thmprime}

The second one focuses on a specific family of admissible and decomposable subsets allowed to build the notion of (asymptotic) expansion, dynamical propagation and quasi-locality. More precisely, we introduce the following refined version of Definition \ref{decomposable}:

\begin{defn}\label{defn:allowed family}
Let $\mathcal{G}$ be a groupoid with a length function $\ell$, and $(\Gz,\RR,\mu)$ be a measure space. Let $\KK$ be a family of admissible bisections which is closed under taking compositions and inverses and $\Gz \in \KK$. A subset $K \subseteq \G$ is called \emph{$\KK$-decomposable} if $K=\bigcup_{i=1}^{N} K_{i}$ for $N \in \N$ and admissible bisection $K_i \in \KK$. 
\end{defn}

Given such a family $\KK$ in Definition \ref{defn:allowed family}, we can replace the word ``decomposable'' by ``$\KK$-decomposable'' in Definition \ref{expanding}, \ref{asymptotically expanding}, \ref{finite dynamical propagation} and \ref{dynamical quasi local} to define the notion of $\KK$-expansion (in measure), $\KK$-asymptotic expansion (in measure), $\KK$-dynamical propagation and $\KK$-dynamical quasi-local, together with $C^*$-algebras $\CC^*_{\dyn}(\G,\KK)$ and  $\CC^*_{\dyn,\q}(\G,\KK)$.

\begin{rem}\label{rem:cofinal family}
Note that if $\KK$ is \emph{cofinal} in the sense that any admissible bisection is contained in the union of finitely many elements in $\KK$, then these notions coincide with the original ones of \emph{possibly different} parameters.
\end{rem}

Following exactly the same proofs, we obtain:

\renewcommand{\thethmpprime}{5.3''}
\begin{thmpprime}\label{cor:cor to thm general case pprime} 
Let $\mathcal{G}$ be a groupoid with a length function $\ell$, $(\Gz,\RR,\mu)$ be a measure space and  $\KK$ be a family of admissible bisections which is closed under taking compositions and inverses and $\Gz \in \KK$. Let $P \in \B(L^2(\Gz,\mu))$ be a rank-one projection, and $\nu$ the associated probability measure on $\Gz$. Then the following are equivalent:
\begin{enumerate}
 \item $P \in \CC^*_{\dyn}(\mathcal{G},\KK)$;
 \item $P \in \CC^*_{\dyn,\q}(\mathcal{G},\KK)$;
 \item $\G$ is $\KK$-asymptotically expanding in measure $\nu$.
\end{enumerate}
\end{thmpprime}

\section{Examples}\label{Examples of measurable groupoids}

In this section, we apply our theory to several classes of groupoids and recover main results in \cite{structure, li2022markovianroealgebraicapproachasymptotic, faucris.282438222}.


\subsection{Transformation groupoids}\label{Group action}

Our first example comes from group actions, which is one of our main motivations for this work.



Let $\rho:\Gamma\curvearrowright X$ be a countable group $\Gamma$ acting on a set $X$, and $\ell_\Gamma$ be a proper word length function on $\Gamma$ in the sense that for any $L>0$, the closed ball denoted by
\[
B_L\coloneqq \{\gamma \in \Gamma~|~\ell_\Gamma(\gamma) \leq L\}
\]
is finite. 
Consider the \emph{transformation groupoid} $X \rtimes \Gamma$ by $\s(x,\gamma) = \gamma^{-1}x$, $\r(x,\gamma)=x$ and $(x,\gamma)^{-1} = (\gamma^{-1}x, \gamma^{-1})$. The length function $\ell_\Gamma$ naturally gives rise to a length function $\ell$ on $X \rtimes \Gamma$ by $\ell(x, \gamma)\coloneqq\ell_{\Gamma}(\gamma)$ for $x\in X$ and $\gamma \in \Gamma$. 


Moreover, let $\mu$ be a measure on $(X,\RR)$ for some $\sigma$-algebra $\RR$ and assume that the action $\rho$ is measure-class-preserving. Then for $\gamma \in \Gamma$, the map $\tau_{X \times \{\gamma\}}: X \to  X$ coincides with $\rho(\gamma)$ (simply denoted by $\gamma$) and $X \times \{\gamma\}$ is admissible.

Concerning decomposable subsets in $X \rtimes \Gamma$, we have the following. The proof is straightforward and hence omitted.







\begin{lem}\label{lem:dec sets for gorup actions}
For any $L>0$, the subset $X \times B_L$ is unital symmetric $|B_L|$-decomposable with $\ell(X \times B_L) \leq L$, and it admits a unital symmetric decomposition $X \times B_L = \bigcup_{\gamma \in B_L} X \times \{\gamma\}$ where each $X \times \{\gamma\}$ is admissible. Here we use $|\cdot|$ to denote the cardinality. 
Conversely, for any decomposable $K \subseteq X \rtimes \Gamma$, we have $K \subseteq X \times B_{\ell(K)}$.
\end{lem}

Hence it suffices to consider decomposable subsets of the form $X \times B_L$, which is determined by its length. Therefore, the notion of (asymptotic) expansion in measure can be translated as follows:

\begin{prop}\label{prop:group action expansion coincidence}
In the above setting, the transformation groupoid $X \rtimes \Gamma$ is asymptotically expanding in measure in the sense of Definition \ref{asymptotically expanding} \emph{if and only if} the action is asymptotically expanding in measure.
A similar result holds for the notion of expansion and domain of (asymptotic) expansion.
\end{prop}

Consequently, Theorem \ref{exhaustion of expansion with radon controlled} and Theorem \ref{thm:Markov structure theorem} recover \cite[Theorem 4.6]{asymptoticexpansionandstrongergodicity} and \cite[Theorem 3.20]{li2022markovianroealgebraicapproachasymptotic}, respectively.

\bigskip

On the other hand, applying Lemma \ref{lem:dec sets for gorup actions} again, the notion of dynamical propagation and quasi-locality can be translated as follows:

\begin{prop}\label{prop:group action ppg and quasi-local coincidence}
In the above setting, an operator $T \in L^2(X,\mu)$ is dynamically quasi-local in the sense of Definition \ref{dynamical quasi local} \emph{if and only if} $T$ is $\rho$-quasi-local in the sense of \cite[Definition 4.3]{li2022markovianroealgebraicapproachasymptotic}, \emph{i.e.}, for any $\epsilon>0$ there exists $L_\epsilon>0$ such that for any measurable $A,B \subseteq X$ with $\mu((B_{L_\epsilon} \cdot A) \cap B)=0$, then $\|\chi_A T \chi_B\| < \epsilon$. A similar result holds for the notion of finite dynamical propagation.
\end{prop}

Consequently, our main result Theorem \ref{thm:main thm ave proj} recovers \cite[Proposition 4.6 and Theorem 4.16]{li2022markovianroealgebraicapproachasymptotic} for the averaging projection when $\mu$ is finite:

\begin{cor}\label{cor:main thm ave proj general gorup action}
Let $\rho:\Gamma\curvearrowright X$ be a countable group $\Gamma$ acting on a probability measure space $(X,\mu)$ and assume that $\rho$ is measure-class-preserving. For the averaging projection $P_X \in \B(L^2(X,\mu))$, the following are equivalent:
\begin{enumerate}
 \item $P$ is $\rho$-quasi-local;
 \item $P$ is a norm limit of operators in $B(L^2(X,\mu))$ with finite $\rho$-propagations;
 \item the action is asymptotically expanding in measure $\mu$.
\end{enumerate}
\end{cor}

Furthermore, we have the following generalised version from Theorem \ref{thm:main thm ave proj general}, which deals with an arbitrary rank-one projection:

\begin{cor}\label{cor:main thm ave proj general gorup action}
Let $\rho:\Gamma\curvearrowright X$ be a countable group $\Gamma$ acting on a (not necessarily finite) measure space $(X,\mu)$ and assume that $\rho$ is measure-class-preserving. Then for any rank-one projection $P \in \B(L^2(X,\mu))$, we have $P$ is $\rho$-quasi-local \emph{if and only if} $P$ is a norm limit of operators in $B(L^2(X,\mu))$ with finite $\rho$-propagations.
\end{cor}

\subsection{Pair groupoids: a family version}\label{main thm for Uniform Roe algebra} 

Now we consider pair groupoids, and aim to recover the results for uniform Roe and quasi-local algebras of metric spaces. 

Firstly, we recall basic notions from coarse geometry and associated $C^*$-algebras.
Let $(X,d)$ be a discrete metric space, $x\in X$ and $R>0$. Denote $B(x,R)$ the \emph{closed ball} with radius $R$ and centre $x$. We say that $(X,d)$ has \emph{bounded geometry} if $\sup_{x\in X} |B(x,R)|$ is finite for each $R>0$. 
More generally, a sequence of metric spaces $\{(X_n,d_n)\}_{n\in \N}$ is said to have \emph{uniformly bounded geometry} if for all $R>0$, the number $\sup_{n\in \N}\sup_{x\in X_n}|B(x,R)|$ is finite. 

\begin{defn}\label{coarse disjoint union}
Let $\{(X_n,d_n)\}_{n\in \mathbb{N}}$ be a sequence of finite metric spaces. A \emph{coarse disjoint union} of $\{(X_n,d_n)\}_{n\in \mathbb{N}}$ is a metric space $(X,d)$, where $X=\bigsqcup_{n\in \mathbb{N}} X_n$ is the disjoint union of $\{X_n\}$ as a set and $d$ is a metric on $X$ such that
\begin{itemize}
    \item the restriction of $d$ on each $X_n$ coincides with $d_n$;
    \item $d(X_n,X_m)\to \infty$ as $n+m\to \infty$ and $n\neq m$.
\end{itemize}
\end{defn}


\begin{defn}\label{define uniform Roe}
Let $(X,d)$ be a discrete metric space with bounded geometry. For an operator $T\in \B(\ell^2(X))$, we say that
\begin{enumerate}
 \item $T$ has finite propagation if there exists $R> 0$ such that for any $A,B\subseteq X$ with $d(A,B)> R$, we have that $\chi_A T\chi_B=0$;
 \item $T$ is \emph{quasi-local} if for any $\epsilon> 0$ there exists $R> 0$ such that for any $A,B\subseteq X$ with $d(A,B)> R$, we have $\|\chi_A T\chi_B\|< \epsilon$.
\end{enumerate}
\end{defn}

The set of all finite propagation operators in $\B(\ell^2(X))$ forms a $\ast$-algebra, called the \emph{algebraic uniform Roe algebra} and denoted by $\C_{\u}[X]$. The \emph{uniform Roe algebra of $X$} is defined to be the operator norm closure of $\C_{\u}[X]$ in $\B(\ell^2(X))$, which is a $C^{\ast}$-algebra and denoted by $\CC^{\ast}_{\u} (X)$. The set of all quasi-local operators in $\B(\ell^2(X))$ forms a $C^{\ast}$-algebra, called the \emph{uniform quasi-local algebra of $X$} and denoted by $\CC^{\ast}_{\uq} (X)$.

\bigskip


In the following, we focus on a coarse disjoint union $(X,d)$ of a sequence of finite metric spaces $\{(X_n,d_n)\}_{n\in \mathbb{N}}$ with uniformly bounded geometry. We will construct a groupoid structure on $X_n$ and apply our main results to study the uniform Roe and quasi-local algebras of $(X,d)$.

To start, for each $n\in \N$ we consider the pair groupoid $X_n \times X_n$ with $(x,y) \cdot (y,z) = (x,z)$, $(x,y)^{-1} = (y,x)$ for $x,y,z \in X_n$. Its unit space can be identified with $X_n$ with source and range maps given by $\s(x,y)=y$ and $\r(x,y) = x$ for $x,y\in X_n$. Moreover, let $\mu_n$ be a finite measure on $(X_n,\RR_n)$ for $\RR_n=\mathcal{P}(X_n)$
and consider the length function $\ell_n$ on $\G_n$ defined by $\ell_n(x,y) \coloneqq d_n(x,y)$ for $x,y\in X_n$. For $L>0$, we denote
\[
E^{(n)}_L\coloneqq \{(x,y) \in X_n \times X_n ~|~ \ell_n(x,y) \leq L\}.
\]

It is clear that any bisection in $\G_n$ is admissible and moreover, we have the following. This is a well-known fact in coarse geometry (see, \emph{e.g.}, \cite[Lemma 12.2.3]{willett2020higher}) due to the assumption that the sequence has uniformly bounded geometry, and hence we omit the proof.

\begin{lem}\label{lem:dec sets for pair groupoids}
For any $L>0$, there exists $N_L \in \N$ such that for any $n\in \N$, the set $E^{(n)}_L$ is unital symmetric and $N_L$-decomposable. Conversely, for any $n\in \N$ and decomposable $K \subseteq \G_n$, we have $K \subseteq E^{(n)}_{\ell(K)}$.
\end{lem}

Hence it suffices to consider decomposable subsets of the form $E^{(n)}_L$. For $A \subseteq X_n$, note that $\r(E^{(n)}_L \cdot A) \setminus A=\partial_{L} A$, where $\partial_{L} A\coloneqq \{x\in X_n \setminus A ~|~ d_n(x, A) \leq L\}$.
Therefore, we can use a family version of Definition \ref{asymptotically expanding} to recover\footnote{Note that although we define the notion of (asymptotic) expansion in the setting of probability measure in Definition \ref{expanding} and \ref{asymptotically expanding}, however, these can be naturally extended to all finite measure spaces as mentioned at the beginning of Section \ref{sec:asymptotic expansion}.} the notion of measured asymptotic expanders introduced in \cite{asymptoticexpansionandstrongergodicity}:

\begin{prop}\label{prop:pair groupoids expansion coincidence}
In the above setting, the following are equivalent:
\begin{enumerate}
 \item there exist functions $\ubar{C}, \ubar{N}, \ubar{L}: (0,\frac{1}{2}] \to (0,\infty)$ such that for all $n\in \N$, $\G_n$ is asymptotically expanding with parameters $\ubar{C}, \ubar{N}, \ubar{L}$;
 \item $\{(X_n, d_n, \mu_n)\}_{n\in \N}$ forms a sequence of measured asymptotic expanders\footnote{When $\mu_n$ is the counting measure on $X_n$, this recovers the class notion of expanders.} in the sense of \cite[Definition 6.1]{asymptoticexpansionandstrongergodicity}, \emph{i.e.}, for any $\alpha \in (0,\frac{1}{2}]$ there exist $C_\alpha>0$ and $L_\alpha>0$ such that for any $n\in \N$ and any $A_n \subseteq X_n$ with $\alpha\mu(X_n) \leq \mu_n(A_n) \leq \frac{1}{2}\mu_n(X_n)$, then $\mu_n(\partial_{L_\alpha} A_n) > C_\alpha \mu_n(A_n)$. 
\end{enumerate}
\end{prop}

Consequently, Theorem \ref{thm:Markov structure theorem} and Remark \ref{rem:dependence 2} recover \cite[Theorem 4.15]{faucris.282438222}.

\bigskip

Now we move to consider finite propagation operators and quasi-local operators. Here we choose $\mu_n$ to be the counting measure on $X_n$. 
From the discussions at the end of Section \ref{ssec:basic notions}, $\G_n$ give rise to a single groupoid $\G = \bigsqcup_{n\in \N} \G_n$ together with a length function $\ell$ and a measure $\mu$ on $\Gz=X$. 
Combining Lemma \ref{family of quasi local and quasi local} and Lemma \ref{lem:dec sets for pair groupoids}, it is easy to see:


%
%


\begin{cor}
Given $T \in \B(\ell^2(X))$, we have $T \in \CC^*_{\dyn,\q}(\G)$ \emph{if and only if} there exists $T_n \in \B(\ell^2(X_n))$ for each $n\in \N$ with $\sup_{n\in \N} \|T_n\| < \infty$ such that $T = (\SOT)-\sum_{n \in \N} T_n$ and for any $\epsilon > 0$ there exists $L>0$ satisfying: for any $n\in \N$ and $A_n,B_n \subseteq X_n$ with $d_n(A_n, B_n) > L$, we have $\|\chi_{A_n} T_n\chi_{B_n}\| < \epsilon$. A similar result holds for $\CC^*_{\dyn}(\G)$.
\end{cor}

We aim to relate $\CC^*_{\dyn}(\G)$ to the uniform Roe algebra $C^*_{\u}(X)$, and $\CC^*_{\dyn,\q}(\G)$ to the uniform quasi-local algebra $C^*_{\uq} (X)$. To achieve, we combine Lemma \ref{family of quasi local and quasi local} and Proposition \ref{prop:pair groupoids expansion coincidence} with \cite[Lemma 3.13 and 3.14]{Bao2021StronglyQA} and reach the following:

\begin{prop}\label{prop:relation between uniform Roe algs}
With the same notation as above, we have
\begin{enumerate}
 \item $\CC^*_{\dyn}(\G) = \CC^*_{\u}(X) \cap \Pi_{n\in \N} \B(\ell^2(X_n))$; 
 \item$\CC^*_{\dyn}(\G) + \K(\ell^2(X)) = \CC^*_{\u}(X)$;
 \item $\CC^*_{\dyn,\q}(\G) = \CC^*_{\uq}(X) \cap \Pi_{n\in \N} \B(\ell^2(X_n))$; 
 \item $\CC^*_{\dyn,\q}(\G) + \K(\ell^2(X)) = \CC^*_{\uq}(X)$.
\end{enumerate}
\end{prop}

Finally, combining with Theorem \ref{cor:cor to thm general case prime} and Proposition \ref{prop:pair groupoids expansion coincidence}, we recover both \cite[Theorem C]{structure} and \cite[Theorem B]{faucris.282438222}:

\begin{cor}\label{cor:example Roe}
Let $\{(X_n, d_n)\}_{n\in \N}$ be a sequence of finite metric spaces with uniformly bounded geometry, and $(X,d)$ be their coarse disjoint union. For each $n\in \N$, let $P_n \in \B(\ell^2(X_n))$ be a rank-one projection and $\nu_n$ the associated probability measure on $X_n$. For $P=\mathrm{(SOT)}-\sum_{n\in \N} P_n \in \B(\ell^2(X))$, the following are equivalent:
\begin{enumerate}
 \item $P \in \CC^*_{\u}(X)$;
 \item $P \in \CC^*_{\uq}(X)$;
 \item $\{(X_n, d_n, \nu_n)\}_{n\in \N}$ forms a sequence of measured asymptotic expanders.
\end{enumerate}
\end{cor}

\subsection{Semi-direct product groupoids}\label{main thm for semi-direct product groupoid}

In this subsection, we consider groupoid actions on fibre spaces, generalising those in Section \ref{Group action}. Our idea is to package the information of groupoid actions into semi-direct product groupoids, which allows us to use our main results to study groupoid actions.


Let us start with some preliminaries for groupoid actions. A \emph{fibre space} over a set $X$ is a pair $(Y,p)$, where $Y$ is a set and $p\colon Y\rightarrow X$ is a surjective map. For two fibre spaces $(Y_1,p_1)$ and $(Y_2,p_2)$ over $X$, we form their \emph{fibred product} to be
\[
    Y_1\tensor*[_{p_1}]{\ast}{_{p_2}}Y_2 \coloneqq \{(y_1, y_2) \in Y_1 \times Y_2 ~|~ p_1(y_1) = p_2(y_2)\}.
\]


\begin{defn}
    Let $\G$ be a groupoid. A \textit{left $\G$-space} is a fibre space $(Y,p)$ over $\mathcal{G}^{(0)}$, equipped with a map $(\gamma,y)\mapsto \gamma y$ from $\G\tensor*[_{\s}]{\ast}{_p}Y$ to $Y$ (called a \emph{$\G$-action on $(Y,p)$}) which satisfies the following:
    \begin{itemize}
        \item $p(\gamma y) = \r(\gamma)$ for $(\gamma,y)\in \G\tensor*[_{\s}]{\ast}{_p}Y$, and $p(y)y = y$;
        \item $\gamma_{2}(\gamma_{1}y) = (\gamma_{2}\gamma_{1})y$ for $(\gamma_{1},y)\in \G\tensor*[_{\s}]{\ast}{_p}Y$ and $\s(\gamma_{2}) = \r(\gamma_{1})$.
    \end{itemize}
\end{defn}

Given a $\G$-space $(Y, p)$, the associated \textit{semi-direct product groupoid} $Y\rtimes\G$ is defined (as a set) to be $Y\tensor*[_{p}]{\ast}{_\r}\G$. For $(y,\gamma) \in Y\rtimes\G$, define its range to be $(y,\r(\gamma))$ and its source to be $(\gamma^{-1}y,\s(\gamma))$. The product and inverse are given by
\[
    (y,\gamma)(\gamma^{-1}y,\gamma')=(y,\gamma\gamma') \quad \text{and} \quad (y,\gamma)^{-1} = (\gamma^{-1}y,\gamma^{-1}).
\]
Clearly $(y,p(y))\mapsto y$ is a bijection from the unit space of $Y\rtimes\G$ onto $Y$. Hence from now on, we regard Y as the unit space of the groupoid $Y\rtimes\G$.

Given a bisection $K \subseteq \G$, we define the following map:
\begin{equation}\label{EQ:beta_K}
\beta_{K}: p^{-1}(\s(K))\rightarrow p^{-1}(\r(K)), \quad y \mapsto \gamma_y\cdot y \quad \text{for} \quad y \in p^{-1}(\s(K)),
\end{equation}
where $\gamma_y\in K$ is the unique point in $K$ such that $\s(\gamma_y)=p(y)$. The following is straightforward:

\begin{lem}\label{lem:semi direct groupoid bisection}
For a groupoid $\G$ acting on a fibre space $(Y,p)$ and a bisection $K \subseteq \G$, then $Y\tensor*[_{p}]{\ast}{_\r}K$ is a bisection in $Y\rtimes\G$. Moreover, we have $\tau_{Y\tensor*[_{p}]{\ast}{_\r}K} = \beta_K$, where $\tau_{Y\tensor*[_{p}]{\ast}{_\r}K}$ is defined in (\ref{EQ:alpha_K}) and $\beta_K$ in (\ref{EQ:beta_K}).
\end{lem}

Analogous to Definition \ref{admissible}, we introduce the following:

\begin{defn}\label{defn:dynamically admissible}
Let $\mathcal{G}$ be a groupoid with a length function $\ell$, and $(Y,p)$ be a $\G$-space equipped with a probability measure structure $(Y,\RR,\mu)$. A bisection $K \subseteq \G$ is called \emph{dynamically admissible} if $p^{-1}(\s(K)),p^{-1}(\r(K)) \in \RR$, the bijection $\beta_{K}$ from (\ref{EQ:beta_K}) is a measure-class-preserving measurable isomorphism and $\ell(K)$ is finite. 
\end{defn}

Now we aim to relate Definition \ref{defn:dynamically admissible} to the notion of admissible bisections for the semi-direct product groupoid $Y \rtimes \G$. Firstly, note that a length function $\ell$ on $\G$ gives rise to a length function $\hat{\ell}$ on $Y \rtimes \G$ by 
\begin{equation}\label{EQ:induced ell}
\hat{\ell}(y,\gamma) \coloneqq \ell(\gamma) \quad \text{for} \quad (y,\gamma)\in Y\rtimes \G.
\end{equation}
Then Lemma \ref{lem:semi direct groupoid bisection} implies the following:

\begin{lem}\label{lem:admissible equivalence between betaK and Yrtimes K}
In the setting above, 
a bisection $K\subseteq \G$ is dynamically admissible \emph{if and only if} $Y\tensor*[_{p}]{\ast}{_\r}K$ is admissible in $Y \rtimes \G$.
\end{lem}


Lemma \ref{lem:admissible equivalence between betaK and Yrtimes K} suggests us to consider the following specific family of admissible bisection in $Y \rtimes \G$:
\begin{equation}\label{EQ:semi-direct family}
\KK_{\G,Y}=\left\{Y \tensor*[_{p}]{\ast}{_\r} K\subseteq Y\rtimes \G~\big|~K\subseteq \G \text{ dynamically admissible}\right\}.
\end{equation}
It is clear that $\KK_{\G,Y}$ is closed under taking composition and inverse, and contains $Y \tensor*[_{p}]{\ast}{_\r} \Gz = Y$ since $\Gz$ is dynamically admissible. Applying the discussion at the end of Section \ref{sec:Main results}, we define $\KK_{\G,Y}$-decomposability for subsets in $Y \rtimes \G$. This can also be translated using the following notion of dynamical decomposability:

\begin{defn}\label{defn:dynamically decomposable}
Let $\mathcal{G}$ be a groupoid with a length function $\ell$, and $(Y,p)$ be a $\G$-space equipped with a probability measure structure $(Y,\RR,\mu)$. A subset $K \subseteq \G$ is called \emph{dynamically decomposable} if $K=\bigcup_{i=1}^{N} K_{i}$ for $N \in \N$ and dynamically admissible bisection $K_i \subseteq \G$.
\end{defn}

Directly from Lemma \ref{lem:admissible equivalence between betaK and Yrtimes K}, we have the following.

\begin{cor}\label{cor:decomposable equivalence between betaK and Yrtimes K}
Let $\mathcal{G}$ be a groupoid with a length function $\ell$, and $(Y,p)$ be a $\G$-space equipped with a probability measure structure $(Y,\RR,\mu)$. Equip $Y\rtimes \G$ with the length function $\hat{\ell}$ from (\ref{EQ:induced ell}). Then for $K\subseteq \G$, $K$ is dynamically decomposable \emph{if and only if} $Y\tensor*[_{p}]{\ast}{_\r}K\subseteq Y\rtimes \G$ is $\KK_{\G,Y}$-decomposable.
\end{cor}

Similar to Lemma \ref{basic assumprop2}, it is easy to see that for any dynamically decomposable $K\subseteq \G$ and measurable $A \subseteq Y$, then $K\cdot A$ is also measurable. Moreover, direct calculations show that 
\[
K\cdot A=\r\big((Y \tensor*[_{p}]{\ast}{_\r} K) \cdot A\big).
\]
Hence it turns out that the $\KK_{\G,Y}$-asymptotic expansion of the semi-direct groupoid $\G \rtimes Y$ is equivalent to the following:


 \begin{defn}\label{asymptotic expansion of groupoid action}
Let $\mathcal{G}$ be a groupoid with a length function $\ell$, and $(Y,p)$ be a $\G$-space equipped with a probability measure structure $(Y,\RR,\mu)$. We say that the $\G$-action on $(Y,p)$ is \emph{asymptotically expanding (in measure $\mu$)} if for any $\alpha\in (0,\frac{1}{2}]$ there exist $C_\alpha, N_\alpha, L_\alpha>0$ and a unital symmetric dynamically $N_\alpha$-decomposable subset $K_\alpha\subseteq \G$ with $\ell(K_\alpha)\leq L_\alpha$ such that for any $A \in \RR$ with $\alpha \leq \mu(A) \leq \frac{1}{2}$, then $\mu((K_\alpha\cdot A)\setminus A)>C_\alpha \mu(A)$. 
 \end{defn}


%
 
Combining Theorem \ref{cor:cor to thm general case pprime} with the discussions above, we reach the following:

\begin{cor}\label{main theorem applied on semi direct groupoid}
Let $\mathcal{G}$ be a groupoid with a length function $\ell$, and $(Y,p)$ be a $\G$-space equipped with a probability measure structure $(Y,\RR,\mu)$. Equip $Y\rtimes \G$ with the length function $\hat{\ell}$ from (\ref{EQ:induced ell}). Let $P$ be a rank-one projection in $\B(L^2(Y,\mu))$ and $\nu$ the associated measure to $P$ constructed in Section \ref{sec:Main results}, then the followings are equivalent:
\begin{enumerate}
	\item $P \in \CC^*_{\dyn}(Y\rtimes \mathcal{G},\KK_{\G,Y})$;
    \item $P \in \CC^*_{\dyn,\q}(Y\rtimes \mathcal{G},\KK_{\G,Y})$;
    \item $Y\rtimes \G$ is $\KK_{\G,Y}$-asymptotically expanding in measure $\nu$;
    \item The $\G$-action on $(Y,p)$ is asymptotically expanding in measure $\nu$.
\end{enumerate}
\end{cor}

\begin{rem}
When $\G$ is a group $G$ acting on a space $Y$, then the action is measure-class-preserving if and only if for any bisection $K \subseteq G$ (in this case, $K$ consists of a single point), $Y\tensor*[_{p}]{\ast}{_\r} K$ is admissible. 
Moreover, if $G$ is equipped with a proper length function, Lemma \ref{lem:dec sets for gorup actions} shows that when the action is measure-class-preserving then the family $\KK_{G,Y}$ is cofinal in the sense of Remark \ref{rem:cofinal family}. Hence combining Remark \ref{rem:cofinal family} with Proposition \ref{prop:group action ppg and quasi-local coincidence}, Corollary \ref{main theorem applied on semi direct groupoid} generalises Corollary \ref{cor:main thm ave proj general gorup action}.
\end{rem}

\subsection{The HLS groupoid and its variant}


Throughout this subsection, let $\Gamma$ be a finitely generated group and $\{N_i\}_{i\in \N}$ be a family of nested, finite index normal subgroups of $\Gamma$ with trivial intersection. For each $i \in \N$, denote the quotient map $\pi_i: \Gamma \to \Gamma/N_i$, and $\pi_\infty: \Gamma \to \Gamma$ the identity map.

The associated \emph{HLS groupoid} (after Higson, Lafforgue and Skandalis) $\G$ is defined to be
\[
\G\coloneqq\bigsqcup_{i\in \N \cup \{\infty\}} \{i\}\times X_i, \quad \text{where} \quad X_i=
\begin{cases}
\Gamma/N_i \quad &\text{if}\quad i\in \N;\\
\Gamma \quad &\text{if}\quad i=\infty.
\end{cases}
\]
Equipping $\G$ with the topology generated by the following sets:
\begin{itemize}
	\item the singletons $\{(i,g)\}$ for $i \in \N$ and $g\in \Gamma/N_i$;
	\item the tails $\{(i,\pi_i (g))~|~i\in \N \cup \{\infty\}, i>N\}$ for $g\in \Gamma$ and $N\in \N$.
\end{itemize}
then $\G$ becomes an \'{e}tale, locally compact and Hausdorff groupoid.

The HLS groupoid was introduced by Higson, Lafforgue and Skandalis in \cite{Higson2002} to provide counterexamples to the groupoid Baum-Connes conjecture. It also provides the first known example of a non-amenable groupoid with the weak containment property due to Willett \cite{Wil15}.

Since the HLS groupoid $\G$ is a bundle of groups, it cannot be asymptotically expanding for any given measure and length function. When regarding $\G$ as a family of groups, then trivially it is uniformly asymptotically expanding since each unit space consists of a single point.


Now we consider a variant of the HLS groupoid from \cite{Alekseev2017}, which provides the first known example of a principal non-amenable groupoid with the weak containment property. To recall their construction, we use the same notation as above and denote $\hat{\Gamma}\coloneqq \varprojlim \Gamma/N_i$ the profinite completion of $\Gamma$ with respect to the family $\{\Gamma/N_i\}_{i \in \N}$. 

For each $i \in \N$, let $\G^{\AFS}_i$ be the transformation groupoid $X_i \rtimes (\Gamma/N_i)$ with the action by left multiplication. For $i = \infty$, let $\G^{\AFS}_\infty$ be the transformation groupoid $\hat{\Gamma} \rtimes \Gamma$ with the natural free action. Then the groupoid constructed in \cite{Alekseev2017} is defined to be the disjoint union (thanks to \cite[Lemma 2.1]{Alekseev2017}):
\[
\G^{\AFS} \coloneqq \bigsqcup_{i \in \N \cup \{\infty\}} \G^{\AFS}_i.
\]

To consider asymptotic expansion, we need to endow measure and length function for each $\G^{\AFS}_i$. For $i\in \N$, we take the normalised counting measure $\mu_i$ on $\Gamma/N_i$. For $i =\infty$, we take the induced probability $\Gamma$-invariant measure $\mu_\infty$ on $\hat{\Gamma}$. On the other hand, fix a length function $\ell$ on $\Gamma$, which induces a quotient length function $\ell_i$ on $\Gamma/N_i$ for each $i \in \N$. Following Section \ref{Group action}, we obtain a length function on $\G^{\AFS}_i$ for each $i \in \N \cup \{\infty\}$. By the discussions in Section \ref{ssec:basic notions}, these can be combined to provide a measure and a length function for $\G^{\AFS}$. 

We have the following characterisation:

\begin{prop}\label{prop:AFS groupoid}
With the same notation as above, the following are equivalent:
\begin{enumerate}
 \item $\G^{\AFS}_i$ is asymptotically expanding in measure uniformly for $i \in \N \cup \{\infty\}$ (\emph{i.e.}, having same expansion parameters);
 \item the natural action of $\Gamma$ on $\hat{\Gamma}$ is asymptotically expanding in measure;
 \item the natural action of $\Gamma$ on $\hat{\Gamma}$ is expanding in measure;
 \item the sequence $\{\Gamma/N_i\}_{i\in \N}$ forms a sequence of asymptotic expander graphs;
 \item the sequence $\{\Gamma/N_i\}_{i\in \N}$ forms a sequence of expander graphs;
 \item $\G^{\AFS}_i$ is expanding in measure uniformly for $i \in \N \cup \{\infty\}$.
\end{enumerate}
\end{prop}

\begin{proof}
``(1) $\Rightarrow$ (2)'': Condition (1) implies that $\G^{\AFS}_\infty$ is asymptotically expanding in measure. Then by Proposition \ref{prop:group action expansion coincidence}, this implies (2).

``(2) $\Leftrightarrow$ (3)'' is due to \cite[Theorem 4]{ABÉRT_ELEK_2012} and \cite[Proposition 3.5]{asymptoticexpansionandstrongergodicity}.

``(3) $\Leftrightarrow$ (5)'' is fairly well-known (see, \emph{e.g.}, \cite[Lemma 2.2]{ABÉRT_ELEK_2012}).

``(2) $\Leftrightarrow$ (4)'' is due to \cite[Theorem 6.16, Corollary 6.17]{asymptoticexpansionandstrongergodicity} and the discussions thereafter. (Note that the sequence $\{\Gamma/N_i\}_{i\in \N}$ is approximating spaces for the action of $\Gamma$ on $\hat{\Gamma}$.)

``(5) $\Rightarrow$ (6)'': For each $i \in \N$, note that $\G^{\AFS}_i$ is isomorphic to the pair groupoid $(\Gamma/N_i) \times (\Gamma/N_i)$, where the length function corresponds to the one given by $(x,y) \mapsto d_i(x,y)$ for the left-invariant metric $d_i$ on $\Gamma/N_i$ induced by $\ell_i$. According to Proposition \ref{prop:pair groupoids expansion coincidence}, (4) is equivalent to that $\G^{\AFS}_i$ is expanding uniformly for $i \in \N$. This concludes (5) since we already showed ``(4) $\Rightarrow$ (3)''.

``(6) $\Rightarrow$ (1)'' is trivial.
\end{proof}

Finally, combining with Theorem \ref{cor:cor to thm general case prime}, we obtain the following:

\begin{cor}
For the groupoid $\G^{\AFS}$, take $P_i \in \B(L^2(X_i,\mu_i))$ to be the averaging projection for $i \in \N \cup \{\infty\}$. For $P=\mathrm{(SOT)}-\sum_{i\in \N \cup \{\infty\}} P_i$, the following are equivalent:
\begin{enumerate}
 \item $P \in \CC^*_{\u}(\G^{\AFS})$;
 \item $P \in \CC^*_{\uq}(\G^{\AFS})$;
 \item the natural action of $\Gamma$ on $\hat{\Gamma}$ is asymptotically expanding in measure;
 \item the sequence $\{\Gamma/N_i\}_{i\in \N}$ forms a sequence of expander graphs.
\end{enumerate}
\end{cor}


\subsection{Graph groupoids}\label{main thm for graph groupoid} 

We firstly recall necessary background knowledge, and details can be found in \cite{KRR97}.

Throughout this subsection, let $G=(V,E,r,s)$ be a directed uniformly locally finite connected graph, where $V$ is the vertex set, $E$ is the edge set and $r,s: E\rightarrow V$ describing the range and source of edges. Here $G$ is called \emph{uniformly locally finite} if $\sup_{v\in V} |s^{-1}(v)|$ and $\sup_{v\in V} |r^{-1}(v)|$ are finite.
Given $v\in V$, denote $\sdeg(v) \coloneqq |s^{-1}(v)|$.


Recall that a \emph{finite path} in $G$ is a sequence $\alpha = (\alpha_1, \cdots, \alpha_k)$ of edges in $E$ with $s(\alpha_{j+1}) = r(\alpha_j)$ for $1\leq j \leq k-1$. Write $s(\alpha) = s(\alpha_1)$ and $r(\alpha) = \r(\alpha_k)$. The \emph{length} of $\alpha$ is defined to be $|\alpha| \coloneqq k$. We denote by $v$ the path of length $0$ with $s(v)=r(v)=v$.
Denote $F(G)$ the set of all finite paths in $G$, and $F(G,v)$ those starting at $v\in V$. Similarly, denote $P(G)$ the set of all infinite paths $\alpha=(\alpha_1,\alpha_2,...)$ in $G$, and $P(G,v)$ those starting at $v\in V$. For $\alpha,\mu\in F(G)$ satisfying $r(\alpha)=s(\mu)$, we define a path $\alpha\mu\in F(G)$ of length $|\alpha|+|\mu|$ by $\alpha\mu=(\alpha_1,...,\alpha_{|\alpha|},\mu_1,...,\mu_{|\mu|})$. We can similarly define $\alpha x\in P(G)$ for $\alpha\in F(G)$ and $x\in P(G)$ satisfying $s(x)=r(\alpha)$. 


Equip $P(G)$ with the induced topology from the product topology on the infinite product $\Pi E$ of the edge set $E$.
By \cite[Corollary 2.2]{KRR97}, this topology is locally compact, $\sigma$-compact, totally disconnected and Hausdorff. Moreover, it has a basis consisting of all the \emph{cylinder sets}:
\[
Z(\alpha)=\left\{x\in P(G) ~|~ x_1=\alpha_1,...,x_{|\alpha|}=\alpha_{|\alpha|}\right\} \quad \text{where} \quad \alpha \in F(G).
\]
By \cite[Lemma 2.1]{KRR97}, $Z(\alpha) \cap Z(\beta) \neq \emptyset$ if and only if either $\alpha$ is a \emph{prefix} of $\beta$ (\emph{i.e.}, $\beta = \alpha \beta'$ for some $\beta' \in F(G)$) or $\beta$ is a prefix of $\alpha$.


To construct the groupoid associated to the graph $G$, we consider the following equivalence relation on $P(G)$: two paths $x,y\in P(G)$ are \emph{shift equivalent} with lag $k\in \mathbb{Z}$ (written $x\sim_k y$) if there exists $N\in \mathbb{N}$ such that $x_i=y_{i+k}$ for all $i\geq N$. 

\begin{defn}\label{defn:graph groupoid}
For a directed uniformly locally finite connected graph $G=(V,E,r,s)$, the associated \emph{graph groupoid} $\G$ is defined to be (as a set):
\[
\mathcal{G}=\{(x,k,y)\in P(G)\times \mathbb{Z}\times P(G) ~|~ x\sim_k y\}
\]
with $(x,k,y)\cdot (y,l,z)\coloneqq(x,k+l,z)$ and $(x,k,y)^{-1}\coloneqq(y,-k,x)$. The range and source maps are given by $\r(x,k,y)=x$ and $\s(x,k,y)=y$, respectively, with $\Gz = P(G)$.
\end{defn}

For $\alpha, \beta \in F(G)$ with $r(\alpha)=r(\beta)$, denote
\[
Z(\alpha,\beta)\coloneqq \left\{(x,k,y) ~|~ x\in Z(\alpha),y\in Z(\beta),k=|\beta|-|\alpha|,x_i=y_{i+k}~\text{for}~i\geq |\alpha|\right\},
\]
which is a bisection in $\G$ with $\r(Z(\alpha, \beta)) = Z(\alpha)$ and $\s(Z(\alpha, \beta)) = Z(\beta)$.
It follows from \cite[Proposition 2.6]{KRR97} that all the sets $Z(\alpha,\beta)$ form a basis for a locally compact Hausdorff topology on $\mathcal{G}$, which makes it a second countable, \'{e}tale groupoid where each $Z(\alpha,\beta)$ is compact open. The induced topology on $\Gz=P(G)$ coincides with the product topology defined above.


\bigskip

To apply our results to the graph groupoid $\G$, we firstly endow $\G$ with the following length function $\ell$. For $(x,k,y) \in \G$, choose the smallest $N \in \N$ such that $x_i=y_{i+k}$ for all $i\geq N$, denoted by $N(x,k,y)$, and set:
\begin{equation}\label{EQ:length function for graph groupoids}
\ell(x,k,y)=2N(x,k,y) + k.
\end{equation}
Note that $N(x,k,y)+k \geq 0$ for any $(x,k,y) \in \G$.
Intuitively, $\ell(x,k,y)$ is the total length of the shortest starting segments in $x$ and $y$, deleting which $x$ and $y$ are the same with lag $k$. 
It is routine to check that $\ell$ is indeed a length function on $\G$ taking values in $\N$. Moreover, given $n \in \N$, the ball in $\G$ defined by $B_n\coloneqq \{(x,k,y) \in \G ~|~ \ell(x,k,y) \leq n\}$ can be reconstructed as follows:
\begin{equation}\label{EQ:dec for ball in graph groupoid}
B_n =\bigcup \left\{Z(\alpha,\beta) ~|~ \alpha, \beta \in F(G) \text{ with } r(\alpha)=r(\beta) \text{ and } |\alpha|+|\beta| \leq n\right\}.
\end{equation}
Moreover, we have 
\begin{equation}\label{EQ:convolution of B1}
B_n = (B_1)^n \quad \text{for any} \quad n\in \N.
\end{equation}
These can be verified by direct calculations, and we omit the details.

Next, we will define a Borel measure $\mu$ on $\Gz=P(G)$. Given $v\in V$ and $b_v>0$, we define a measure $\mu_v$ on $P(G,v)$ by firstly setting:
\begin{equation}\label{EQ:def for mu_v}
\mu_v(Z(\alpha)) \coloneqq b_v \cdot \prod_{i=1}^{k} \frac{1}{\sdeg(s(\alpha_i))} \quad \text{for} \quad \alpha = (\alpha_1, \alpha_2, \cdots, \alpha_k) \in F(G,v).
\end{equation}
Since the sets $Z(\alpha)$ for $\alpha \in F(G,v)$ form a basis for the topology on $P(G,v)$, then the above gives rise to a Borel measure $\mu_v$ on $P(G,v)$.

Now we combine $\mu_v$ into a single probability Borel measure $\mu$ on $P(G)$. Since $G$ is connected and uniformly locally finite, then $V$ is countable. For each $v\in V$, we choose $b_v>0$ such that $\sum_{v\in V} b_v =1$. Note that $P(G) = \bigsqcup_{v\in V} P(G,v)$ and $V$ is countable. Then the measures $\mu_v$ defined above give rise to a Borel measure $\mu$ on $P(G)$. In the sequel, we fix such a measure $\mu$.


By the construction of $\ell$ and $\mu$ above, it is clear that for any $\alpha, \beta \in F(G)$ with $r(\alpha)=r(\beta)$, the subset $Z(\alpha, \beta)$ is an admissible bisection. Moreover, we have:

\begin{lem}\label{decomposable and bounded length on graph groupoid}
For the graph groupoid $\G$ with the length function $\ell$ and the measure $\mu$ on $\Gz$ above, the ball $B_n=\{(x,k,y) \in \G ~|~ \ell(x,k,y) \leq n\}$ is decomposable for each $n\in \N$. 
\end{lem}

\begin{proof}
Fix $n\in \N$. We aim to decompose the ball $B_n$ into finitely many admissible bisections. The idea is to combine several $Z(\alpha,\beta)$ in (\ref{EQ:dec for ball in graph groupoid}) into a single bisection. 

To achieve, we construct another graph $\hat{G} = (\hat{V}, \hat{E})$ with 
\[
\hat{V} \coloneqq \big\{(\alpha, \beta) ~\big|~ \alpha, \beta \in F(G) \text{ with } r(\alpha)=r(\beta) \text{ and } |\alpha|+|\beta| \leq n\big\}
\]
and two vertices $(\alpha_1, \beta_1)$ and $(\alpha_2, \beta_2)$ in $\hat{V}$ are connected by an edge if and only if at least one of the following four cases hold: (i) $\alpha_1$ is a prefix of $\alpha_2$; (ii) $\alpha_2$ is a prefix of $\alpha_1$; (iii) $\beta_1$ is a prefix of $\beta_2$; (iv) $\beta_2$ is a prefix of $\beta_1$. 
Since $G$ is uniformly locally finite and connected, then $\hat{G}$ is uniformly locally finite as well. It is a well-known fact (see, \emph{e.g.}, \cite[Lemma 4.2]{ZhangquasilocalityandpropertyA}) that $\hat{V}$ can be decomposed into $\hat{V} = \bigsqcup_{p=1}^N \hat{V}_p$ such that for each $p$, any two distinct vertices in $\hat{V}_p$ are \emph{not} connected by an edge. 

For each $p=1,\cdots,N$, set
\[
K_p\coloneqq \bigsqcup_{(\alpha ,\beta) \in \hat{V}_p} Z(\alpha, \beta).
\]
By the construction of $\hat{E}$, it is routine to check that $K_p$ is an admissible bisection. Finally, (\ref{EQ:dec for ball in graph groupoid}) implies $B_n = \bigcup_{p=1}^N K_p$ and hence, we conclude the proof.
\end{proof}

Lemma~\ref{decomposable and bounded length on graph groupoid} shows that the set $\{B_n~|~n\in \N\}$ is cofinal in the set of all decomposable subsets of $\G$ in the sense of Remark \ref{rem:cofinal family}. Hence we can merely use these balls to define asymptotic expansion, dynamical propagation and quasi-locality for graph groupoids by the discussion at the end of Section \ref{sec:Main results}.

Moreover, we have the following:


\begin{lem}\label{replace Bl(k) by Bl(1)}
For the graph groupoid $\G$ with the length function $\ell$ and the measure $\mu$ on $\Gz$ above, then $\mathcal{G}$ is asymptotically expanding in measure \emph{if and only if} for any $\alpha\in (0,\frac{1}{2}]$, there exists $C_\alpha>0$ such that for any measurable $A\subseteq \mathcal{G}^{(0)}$ with $\alpha \leq \mu(A)\leq \frac{1}{2}$, then $\mu(\r(B_1 \cdot A))>(1+C_\alpha)\mu(A)$.
\end{lem}

Thanks to Lemma \ref{decomposable and bounded length on graph groupoid} and (\ref{EQ:convolution of B1}), the proof follows almost the same as that for \cite[Lemma 3.16]{asymptoticexpansionandstrongergodicity}, and hence we omit the proof. In fact, a similar result (\emph{i.e.}, we can always choose the decomposable subset $K$ in the definition of asymptotic expansion to be the unit ball $B_1$) holds provided each ball $B_n=\{\gamma \in \G ~|~ \ell(\gamma) \leq n\}$ is decomposable and $B_n = (B_1)^n$ for any $n\in \N$.

Lemma \ref{replace Bl(k) by Bl(1)} can be further refined as follows under an extra assumption:

\begin{prop}\label{simplify A into finite sets}
Let $G$ be a uniformly locally finite connected graph, and $\G$ the associated graph groupoid with the length function $\ell$ and the measure $\mu$ on $\Gz$ above. Moreover, assume that there exists $C\geq 1$ such that $\frac{b_w}{C} \leq b_v \leq Cb_w$ for any $v,w \in V$ connected by an edge, where $(b_v)_v$ are the weights used to define $\mu$. Then $\G$ is asymptotically expanding in measure \emph{if and only if} for any $0<\alpha\leq\frac{1}{2}$, there exists $C_\alpha>0$ such that for any measurable $A \subseteq \Gz$ of the form $A=\bigcup_{i=1}^{N} Z(\alpha_i)$ with $\alpha \leq\mu(A)\leq\frac{1}{2}$ for $N\in \N$ and $\alpha_i\in F(G)$, we have $\mu(\r(B_1 \cdot A))>(1+C_\alpha)\mu(A)$.
\end{prop}


\begin{proof}
By Lemma \ref{replace Bl(k) by Bl(1)}, we only need to prove sufficiency. Since $\Gz = P(G)$ is a locally compact Hausdorff space in which every open subset is $\sigma$-compact and $\mu$ is finite, it follows from \cite[Theorem 7.8]{Fol99} that $\mu$ is regular. Take $C$ as in the assumption and set
\[
D=\max\left\{\sup_{v\in V} |s^{-1}(v)|, \, \sup_{v\in V} |r^{-1}(v)|\right\}.
\]

We claim that for any measurable $A' \subseteq P(G)$, we have 
\begin{equation}\label{EQ:upper bound for expansion}
\mu(\r(B_1 \cdot A')) \leq CD(D+2)\mu(A').
\end{equation}
Indeed, for any $\alpha = (\alpha_1, \cdots, \alpha_k) \in F(G)$ with $k \geq 1$, set
\[
F(G)_\alpha \coloneqq \{\alpha\} \cup \{(\alpha_2, \cdots, \alpha_k)\} \cup \{(\alpha_0, \alpha_1,\cdots,\alpha_k) ~|~ \alpha_0 \in E \text{ with } r(\alpha_0) = s(\alpha_1)\}.
\]
Here we take the convention that if $k=1$, then $(\alpha_2, \cdots, \alpha_k) = r(\alpha_1)$. By the choice of $D$, we have $|F(G)_\alpha| \leq D+2$. Moreover, for any $\beta \in F(G)_\alpha$, we have
\[
\mu(Z(\beta)) \leq CD \mu(Z(\alpha))
\]
by the definition of $\mu$ (see (\ref{EQ:def for mu_v})). Hence we obtain:
\[
\mu(\r(B_1 \cdot Z(\alpha))) \leq \sum_{\beta \in F(G)_\alpha} \mu(Z(\beta)) \leq CD(D+2)\mu(Z_\alpha).
\]

Denote $\mathcal{M}\coloneqq \{\text{measurable }A' \subseteq P(G) ~|~ (\ref{EQ:upper bound for expansion}) \text{ holds for } A'\}$. Note that the set $\mathcal{P}\coloneqq \{Z(\alpha) ~|~ \alpha \in F(G)\}$ is a semi-ring and hence, the analysis above implies that $\mathcal{M}$ contains the ring $R(\mathcal{P})$ generated by $\mathcal{P}$. Due to the Monotone Class Theorem, we know that $\mathcal{M}$ contains the $\sigma$-algebra generated by $\mathcal{P}$, which coincides with the Borel $\sigma$-algebra on $P(G)$. Hence we conclude the claim.


Now given $0<\alpha\leq \frac{1}{2}$, we have a constant $C_{\alpha/2}>0$ from the assumption for $\alpha/2$. Also given a measurable $A \subseteq P(G)$ with $\alpha\leq \mu(A)<\frac{1}{2}$, take 
\[
\epsilon_0 \coloneqq \min\left\{\frac{\alpha}{2}, \frac{1}{2}-\mu(A), \frac{\alpha C_{\alpha/2}}{2+2C_{\alpha/2} + 2CD(D+2)}\right\} > 0. 
\]
Due to regularity, there is an open subset $\tilde{A} = \bigcup_{i=1}^{N} Z(\alpha_i)$ for $\alpha_i \in F(G)$ such that $\mu(A \symdiff \tilde{A}) < \epsilon_0$, which implies that $\frac{\alpha}{2} \leq \mu(\tilde{A}) \leq \frac{1}{2}$. By the claim above, we have:
\begin{align}\label{EQ:estimate in graph groupoids}
\mu(\r(B_1 \cdot A))&  \geq  \mu(\r(B_1 \cdot \tilde{A})) - \mu(\r(B_1 \cdot (A \symdiff \tilde{A}))) 
>(1+C_{\alpha/2})\mu(\tilde{A}) - CD(D+2)\mu(A \symdiff \tilde{A}) \nonumber\\
& \geq (1+C_{\alpha/2})(\mu(A) - \epsilon_0) - CD(D+2) \epsilon_0  \nonumber \\
& \geq (1+C_{\alpha/2})\mu(A) - \frac{\alpha C_{\alpha/2}}{2} \geq (1+\frac{C_{\alpha/2}}{2})\mu(A).
\end{align}

Finally, we consider measurable $A \subseteq P(G)$ with $\mu(A) = \frac{1}{2}$. Take
\[
\epsilon_1 \coloneqq \min\left\{\frac{3}{8}, ~\frac{C_{1/8}}{32+4C_{1/8} + 32CD(D+2)}\right\} > 0. 
\]
Without loss of generality, assume $C_{1/8} \leq 1$. Again due to regularity, we choose open subset $\tilde{A} = \bigcup_{i=1}^{N} Z(\alpha_i)$ for $\alpha_i \in F(G)$ such that $\mu(A \symdiff \tilde{A}) < \epsilon_1$. This implies that $ \frac{1}{2} - \epsilon_1 < \mu(\tilde{A}) < \frac{1}{2} + \epsilon_1$. We need to control the expansion of $\tilde{A}$.

If $\mu(\tilde{A}) \leq \frac{1}{2}$, then by assumption we have $\mu(\r(B_1\cdot \tilde{A})\setminus \tilde{A})\geq C_{1/8} \mu(A)$. Otherwise, set $\tilde{B} \coloneqq P(G)\setminus (\r(B_1\cdot \tilde{A}))$ and then $\mu(\tilde{B}) < \frac{1}{2}$. If $\mu(\tilde{B})< \frac{1}{4}$, then $\mu(\r(B_1\cdot \tilde{A})) = 1- \mu(\tilde{B}) >1-\frac{1}{4} \geq \frac{3}{2}\mu(A)$. If $\mu(\tilde{B})\geq \frac{1}{4}$, then applying the estimate (\ref{EQ:estimate in graph groupoids}) to $\tilde{B}$ we obtain $\mu(\r(B_1\cdot \tilde{B}))> (1+\frac{C_{1/8}}{2})\mu(\tilde{B})$. Hence 
\[
\mu(\r(B_1\cdot \tilde{A})\setminus \tilde{A})\geq \mu(\r(B_1\cdot \tilde{B})\setminus\tilde{B})> \frac{C_{1/8}}{2} \mu(\tilde{B})\geq \frac{C_{1/8}}{8} \cdot \mu(\tilde{A}).
\]
In conclusion, we obtain $\mu(\r(B_1\cdot \tilde{A})) > (1+ \frac{C_{1/8}}{8}) \mu(\tilde{A})$. Then we obtain:
\begin{align*}
\mu(\r(B_1 \cdot A))&  \geq  \mu(\r(B_1 \cdot \tilde{A})) - \mu(\r(B_1 \cdot (A \symdiff \tilde{A}))) 
>(1+\frac{C_{1/8}}{8})\mu(\tilde{A}) - CD(D+2)\mu(A \symdiff \tilde{A}) \nonumber\\
& \geq (1+\frac{C_{1/8}}{8})(\mu(A) - \epsilon_1) - CD(D+2) \epsilon_1  \nonumber \\
& \geq (1+\frac{C_{1/8}}{8})\mu(A) - \frac{C_{1/8}}{8} = (1+\frac{C_{1/8}}{16})\mu(A),
\end{align*}
which concludes the proof.
\end{proof}

Finally, we provide a concrete graph and study the asymptotic expansion of its graph groupoid.

\begin{ex}\label{example on the graph of N with k period}
Fix $k$ be a positive integer. Consider the graph $G=(V,E)$ where $V= \N$ and $E=\{(n,n+i) ~|~ n\in \N \text{ and } i=1,2,\cdots, k\}$. Let $\G$ be the associated graph groupoid with the length function from (\ref{EQ:length function for graph groupoids}), and equip $P(G)$ with the Borel probability measure from (\ref{EQ:def for mu_v}) by setting $b_n=\mu(Z(n)) \coloneqq \frac{1}{2^{n+1}}$ for $n\in \N$. More precisely, for $\alpha=(\alpha_1, \cdots, \alpha_m) \in F(G)$ with $\s(\alpha_1) = n$, set $\mu(Z(\alpha)) = \frac{1}{2^{n+1}k^m}$.
%
%
We have the following:
\begin{enumerate}
\item If $k=1$, then $\G$ is expanding in measure.
\item If $k\geq 2$, then $\G$ is not asymptotically expanding in measure.
\end{enumerate}

\textbf{For (1):} Note that for any $\alpha \in F(G)$, the set $Z(\alpha)$ consists of a single point, which is the infinite path starting at $s(\alpha)$. Hence $Z(\alpha) = Z(s(\alpha))$. Now for any measurable $A \subseteq P(G)$ with $0<\mu(A) \leq \frac{1}{2}$, take $n_A\coloneqq \min\{s(\alpha) ~|~ \alpha \in A\}$. 

If $n_A=0$, then $\mu(A) \leq \frac{1}{2}$ implies that $A=Z(0)$. Hence $Z(1)\subseteq \r(B_1\cdot A)$, which implies that $\mu(\r(B_1\cdot A))\geq \frac{3}{2}\mu(A)$. 
If $n_A\geq 1$, then by the choice of $\mu$ we have $\mu(Z(n_A))\geq \frac{1}{2}\mu(A)$. Since $Z(n_A-1)\subseteq \r(B_1\cdot A)$, then 
\[
\mu(\r(B_1\cdot A))\geq \mu(A)+2\mu(Z(n_A))\geq 2\mu(A). 
\]
Hence $\mathcal{G}$ is expanding in measure, and we conclude (1).

\textbf{For (2):} Fix $k\geq 2$. For $n_1 \leq n_2\in \N$, denote
\[
F(G)_{n_1, n_2} \coloneqq \{\alpha \in F(G) ~|~ s(\alpha) = n_1, r(\alpha) = n_2\} \quad \text{and} \quad Z_{n_1, n_2} \coloneqq \bigcup \{Z(\alpha) ~|~ \alpha \in F(G)_{n_1, n_2}\}.
\]
Direct calculations show that for $n \geq k$, we have
\[
\mu(Z_{0,n})=\mu(Z(0)) -\sum_{j=1}^{k-1}\frac{k-j}{k}\mu(Z_{0,n-j}).
\]
Rewrite the above as follows:
\[
\mu(Z_{0,n}) + \sum_{j=1}^{k-1}\frac{k-j}{k}\mu(Z_{0,n-j}) = \frac{1}{2},
\]
and note that $1+\sum_{j=1}^{k-1}\frac{k-j}{k} = \frac{k+1}{2}$. Hence for any $p \in \N$, there exists an integer $n_p \in \left\{pk+1, pk+2, \cdots, (p+1)k\right\}$ such that $\mu(Z_{0,n_p})\geq\frac{1}{k+1}$.

Now for positive integer $p \in \N$, define
\[
F(G)_{n_p} \coloneqq \bigsqcup_{m=0}^{n_p} F(G)_{m,n_p} \quad \text{and} \quad A_p=\bigcup \left\{ Z(\alpha\alpha_p)~|~\alpha \in F(G)_{n_p}\right\},
\]
where $\alpha_p\coloneqq (n_p,n_p+1)$. Then we have 
\[
\mu(A_p)>\frac{1}{k}\mu(Z_{0,n_p})\geq \frac{1}{k(k+1)} \quad \text{and} \quad \mu(A_p)\leq \frac{1}{k}\mu(P(G))=\frac{1}{k}.
\]
By the construction of $F(G)_{n_p}$, it is clear that for any $\alpha \in F(G)_{n_p}$ with $|\alpha|\geq 1$ we have $\r(B_1\cdot Z(\alpha\alpha_p))\subseteq A_p$. Hence
\[
\mu(\r(B_1\cdot A_p))=\mu(A_p \sqcup Z(n_p +1))=\mu(A_p)+\frac{1}{2^{n_p +2}},
\]
which shows that $\mu\big(\r((B_1\cdot A_p) \setminus A_p)\big) = \frac{1}{2^{n_p +2}} \to 0$ as $p \to \infty$. This concludes (2).
\end{ex}

\bibliographystyle{plain}
\bibliography{reference}

\begin{thebibliography}{10}

\bibitem{ABÉRT_ELEK_2012}
Mikl\'{o}s Ab\'{e}rt and G\'{a}bor Elek.
\newblock Dynamical properties of profinite actions.
\newblock {\em Ergod. Theory Dyn. Syst.}, 32(6):1805–1835, 2012.

\bibitem{Alekseev2017}
Vadim Alekseev and Martin Finn-Sell.
\newblock {Non-Amenable} {Principal} {Groupoids} with {Weak} {Containment}.
\newblock {\em Int. Math. Res. Not. IMRN}, 2018(8):2332--2340, 01 2017.

\bibitem{Bao2021StronglyQA}
Hengda Bao, Xiaoman Chen, and Jiawen Zhang.
\newblock Strongly quasi-local algebras and their {$K$}-theories.
\newblock {\em J. Noncommut. Geom.}, 17(1):241--285, 2023.

\bibitem{BCH94}
Paul Baum, Alain Connes, and Nigel Higson.
\newblock Classifying space for proper actions and {$K$}-theory of group {$\rm C^\ast$}-algebras.
\newblock In {\em {$C^\ast$}-algebras: 1943--1993 ({S}an {A}ntonio, {TX}, 1993)}, volume 167 of {\em Contemp. Math.}, pages 240--291. Amer. Math. Soc., Providence, RI, 1994.

\bibitem{BdS16}
Yves Benoist and Nicolas de~Saxc\'e.
\newblock A spectral gap theorem in simple {L}ie groups.
\newblock {\em Invent. Math.}, 205(2):337--361, 2016.

\bibitem{BG08}
Jean Bourgain and Alex Gamburd.
\newblock On the spectral gap for finitely-generated subgroups of {$\rm SU(2)$}.
\newblock {\em Invent. Math.}, 171(1):83--121, 2008.

\bibitem{brownronald1987}
Ronald Brown.
\newblock From groups to groupoids: a brief survey.
\newblock {\em Bull. London Math. Soc.}, 19(2):113--134, 1987.

\bibitem{CW80}
Alain Connes and Benjamin Weiss.
\newblock Property {${\rm T}$}\ and asymptotically invariant sequences.
\newblock {\em Israel J. Math.}, 37(3):209--210, 1980.

\bibitem{delaat2023dynamicalpropagationroealgebras}
Tim de~Laat, Federico Vigolo, and Jeroen Winkel.
\newblock Dynamical propagation and {Roe} algebras of warped spaces.
\newblock {\em arxiv preprint {arXiv}:2308.07006}, 2023.

\bibitem{Engel2015IndexTO}
Alexander Engel.
\newblock Index theorems for uniformly elliptic operators.
\newblock {\em New York J. Math.}, 24:543--587, 2018.

\bibitem{Engel2019RoughIT}
Alexander Engel.
\newblock Rough index theory on spaces of polynomial growth and contractibility.
\newblock {\em J. Noncommut. Geom.}, 13(2):617--666, 2019.

\bibitem{Fol99}
Gerald~Budge Folland.
\newblock {\em Real analysis}.
\newblock Pure Appl. Math. (New York). John Wiley \& Sons, Inc., New York, second edition, 1999.
\newblock Modern techniques and their applications, A Wiley-Interscience Publication.

\bibitem{GJS99}
Alex Gamburd, Dmitry Jakobson, and Peter Sarnak.
\newblock Spectra of elements in the group ring of {${\rm SU}(2)$}.
\newblock {\em J. Eur. Math. Soc. (JEMS)}, 1(1):51--85, 1999.

\bibitem{Higgins1971}
Philip~John Higgins.
\newblock {\em Notes on categories and groupoids}, volume No. 32 of {\em Van Nostrand Reinhold Math. Stud.}
\newblock Van Nostrand Reinhold Co., London-New York-Melbourne, 1971.

\bibitem{Higson2002}
Nigel Higson, Vincent Lafforgue, and Georges Skandalis.
\newblock Counterexamples to the {Baum-Connes} conjecture.
\newblock {\em Geom. Funct. Anal.}, 12:330--354, 06 2002.

\bibitem{quasilocalityforetalegroupoid}
Baojie Jiang, Jiawen Zhang, and Jianguo Zhang.
\newblock Quasi-locality for \'etale groupoids.
\newblock {\em Comm. Math. Phys.}, 403(1):329--379, 2023.

\bibitem{structure}
Ana Khukhro, Kang Li, Federico Vigolo, and Jiawen Zhang.
\newblock On the structure of asymptotic expanders.
\newblock {\em Adv. Math.}, 393:Paper No. 108073, 35, 2021.

\bibitem{KRR97}
Alex Kumjian, David Pask, Iain Raeburn, and Jean Renault.
\newblock Graphs, groupoids, and {C}untz-{K}rieger algebras.
\newblock {\em J. Funct. Anal.}, 144(2):505--541, 1998.

\bibitem{lawler1988bounds}
Gregory~Francis Lawler and Alan~David Sokal.
\newblock Bounds on the {$L^2$} spectrum for {Markov} chains and {Markov} processes: a generalization of {Cheeger’s} inequality.
\newblock {\em Trans. Am. Math. Soc.}, 309(2):557--580, 1988.

\bibitem{Li2019QuasilocalAA}
Kang Li, Piotr Nowak, J\'{a}n \v{S}pakula, and Jiawen Zhang.
\newblock Quasi-local algebras and asymptotic expanders.
\newblock {\em Groups Geom. Dyn.}, 15(2):655--682, 2021.

\bibitem{asymptoticexpansionandstrongergodicity}
Kang Li, Federico Vigolo, and Jiawen Zhang.
\newblock Asymptotic expansion in measure and strong ergodicity.
\newblock {\em J. Topol. Anal.}, 15(2):361--399, 2023.

\bibitem{li2022markovianroealgebraicapproachasymptotic}
Kang Li, Federico Vigolo, and Jiawen Zhang.
\newblock A {M}arkovian and {R}oe-algebraic approach to asymptotic expansion in measure.
\newblock {\em Banach J. Math. Anal.}, 17(4):Paper No. 74, 54, 2023.

\bibitem{faucris.282438222}
Kang Li, J\'{a}n \v{S}pakula, and Jiawen Zhang.
\newblock Measured asymptotic expanders and rigidity for {R}oe algebras.
\newblock {\em Int. Math. Res. Not. IMRN}, (17):15102--15154, 2023.

\bibitem{Lub12}
Alexander Lubotzky.
\newblock Expander graphs in pure and applied mathematics.
\newblock {\em Bull. Amer. Math. Soc. (N.S.)}, 49(1):113--162, 2012.

\bibitem{ozawa2024embeddingsmatrixalgebrasuniform}
Narutaka Ozawa.
\newblock Embeddings of matrix algebras into uniform {Roe} algebras and quasi-local algebras.
\newblock {\em arxiv preprint {arXiv}:2310.03677}, 2024.

\bibitem{Rev75}
Daniel Revuz.
\newblock {\em Markov chains}, volume~11 of {\em North-Holland Math. Libr.}
\newblock North-Holland Publishing Co., Amsterdam, second edition, 1984.

\bibitem{Roe1988AnIT}
John Roe.
\newblock An index theorem on open manifolds. {I}, {II}.
\newblock {\em J. Differential Geom.}, 27(1):87--113, 115--136, 1988.

\bibitem{sawicki2017warped}
Damian Sawicki.
\newblock Warped cones violating the coarse {B}aum-{C}onnes conjecture.
\newblock {\em preprint}, 2017.

\bibitem{Sch80}
Klaus Schmidt.
\newblock Asymptotically invariant sequences and an action of {${\rm SL}(2,\,{\bf Z})$}\ on the {$2$}-sphere.
\newblock {\em Israel J. Math.}, 37(3):193--208, 1980.

\bibitem{Sch81}
Klaus Schmidt.
\newblock Amenability, {K}azhdan's property {$T$}, strong ergodicity and invariant means for ergodic group-actions.
\newblock {\em Ergodic Theory Dynam. Systems}, 1(2):223--236, 1981.

\bibitem{ZhangquasilocalityandpropertyA}
J{\'a}n {\v S}pakula and Jiawen Zhang.
\newblock Quasi-locality and {P}roperty {A}.
\newblock {\em J. Funct. Anal.}, 278(1):108299, 2020.

\bibitem{Vig19}
Federico Vigolo.
\newblock Measure expanding actions, expanders and warped cones.
\newblock {\em Trans. Amer. Math. Soc.}, 371(3):1951--1979, 2019.

\bibitem{Wil15}
Rufus Willett.
\newblock A non-amenable groupoid whose maximal and reduced {$C^*$}-algebras are the same.
\newblock {\em M\"{u}nster J. Math.}, 8(1):241--252, 2015.

\bibitem{willett2020higher}
Rufus Willett and Guoliang Yu.
\newblock {\em Higher index theory}, volume 189.
\newblock Cambridge University Press, 2020.

\bibitem{Yu00}
Guoliang Yu.
\newblock The coarse {B}aum-{C}onnes conjecture for spaces which admit a uniform embedding into {H}ilbert space.
\newblock {\em Invent. Math.}, 139(1):201--240, 2000.

\end{thebibliography}

\end{document}